\documentclass{article}
\usepackage[a4paper, total={6.2in, 9.6in}]{geometry}
\usepackage{tikz}
\usepackage[utf8]{inputenc}
\usepackage{amsmath,amsfonts,amsthm,amssymb, amsopn,mathtools}
\usepackage[final, pdftex, pdfpagelabels, pdfstartview = {FitH}, bookmarks, colorlinks, plainpages = false, linktoc=all, linkcolor=red, citecolor=blue, urlcolor=blue,filecolor=black]{hyperref}
\usepackage{xcolor}
\usepackage[textwidth=20mm]{todonotes}
\usepackage{cleveref}
\usepackage{pdfrender, tikz-cd, rotating}
\usepackage{fancyhdr}
\usepackage{enumitem}
\usepackage{mathtools}
\usepackage{titlesec}

\titleformat{\subsection}[runin]
  {\normalfont\bfseries}{\thesubsection.}{0em}{ }

\usepackage{etoolbox, array, setspace}
\usepackage{url}
\newcounter{rowcntr}[table]
\renewcommand{\therowcntr}{\arabic{rowcntr}}
\newcolumntype{N}{>{\refstepcounter{rowcntr}\therowcntr}c}
\AtBeginEnvironment{table}{\setcounter{rowcntr}{0}}

\numberwithin{equation}{section}

\newcommand*{\isoarrow}[1]{\arrow[#1,"\rotatebox{90}{\(\sim\)}"]}
\tikzcdset{ cells={font=\everymath\expandafter{\the\everymath\displaystyle}}, }

\theoremstyle{plain}
\newtheorem{theorem}{Theorem}[section]
\newtheorem{lemma}[theorem]{Lemma}
\newtheorem{prop}[theorem]{Proposition}
\newtheorem{cor}[theorem]{Corollary}

\newtheorem{conj}[theorem]{Conjecture}

\theoremstyle{definition}
\newtheorem{definition}[theorem]{Definition}
\newtheorem{ex}[theorem]{Example}
\newtheorem{remark}[theorem]{Remark}

\theoremstyle{remark}

\newcommand{\Nhat}{\hat{\mathbf{N}}}
\newcommand{\bbZ}{\mathbb{Z}}
\newcommand{\bbQ}{\mathbb{Q}}
\newcommand{\bbN}{\mathbb{N}}

\newcommand{\bbR}{\mathbb{R}}

\newcommand{\bfN}{\mathbf{N}}
\newcommand{\bfM}{\mathbf{M}}
\newcommand{\bfH}{\mathbf{H}}
\newcommand{\frX}{\mathfrak{X}}
\newcommand{\undH}{\underline{\bfH}}
\newcommand{\calH}{\mathcal{H}}
\newcommand{\ra}{\rightarrow}
\newcommand{\Htil}{\widetilde{\calH}}
\newcommand{\Pgeq}[1]{\Phi^{\geq #1}}
\newcommand{\mcL}{\mathcal{L}}
\newcommand{\bfMt}{\bfM^{\Pgeq{3}}}

\usepackage{}\newcommand{\HTIL}{\widetilde{\underline{\bfH}}}

\DeclareMathOperator{\kpf}{kpf}
\DeclareMathOperator{\hgt}{ht}

\title{Pre-canonical bases on affine Hecke algebras}
\author{Nicolas Libedinsky, Leonardo Patimo, David Plaza }
\date{}

\begin{document}
\maketitle
\begin{abstract}
For any affine Weyl group, we introduce the pre-canonical bases. They are a set of bases $\{\bfN^i\}_{1\leq i\leq m+1}$ (where $m$ is the height of the highest root) of the spherical Hecke algebra  that interpolates between the standard basis $\bfN^1$ and the canonical basis $\bfN^{m+1}$. The expansion of  $\bfN^{i+1}$ in terms of the $\bfN^i$ is in many cases very simple and we conjecture that in type $A$ it is positive. 
\end{abstract}


\section{Introduction}

This paper introduces the notion of pre-canonical bases on the spherical Hecke algebra. 
  The motivation comes from the study of Kazhdan-Lusztig polynomials in representation theory, and especially from  Sch\"{u}tzer's work on character formulas for Lie groups. From a computational point of view, the most interesting feature of the definition is that these bases interpolate between the standard and the canonical bases, thus dividing the hard problem of calculating Kazhdan-Lusztig polynomials (or $q$-analogues of weight multiplicities) into a finite number of much easier problems.

\subsection{Overlook.} \label{a3} Before going any further let us roughly describe a simple example of the pre-canonical bases on a spherical Hecke algebra. Let $\Phi$ be a root system of type $A_3$ with simple roots $\{\alpha_1, \alpha_2, \alpha_3\}$. Let  $X^+$ be the set of dominant weights and let $\Htil$ be the spherical Hecke algebra with scalars $\mathbb{Z}[q^{\frac{1}{2}}, q^{-\frac{1}{2}}]$. The definition of $\Htil$ will come later; for now  we just need to know that  $\Htil$ has a standard basis $\{\bfH_{\lambda} : \lambda\in X^+\}$ and a canonical (or Kazhdan-Lusztig) basis $\{\underline{\bfH}_{\lambda} : \lambda\in X^+\}$.  In this case there are four pre-canonical bases, namely \[\{\bfN_{\lambda}^{i} : \lambda\in X^+\}\ \ \ \mathrm{for}\  1\leq i\leq 4,\]
with $ \bfN_{\lambda}^4$ being the canonical basis $\underline{\bfH}_{\lambda }$ and $ \bfN_{\lambda}^1$ being the standard basis ${\bfH}_{\lambda }. $
Let  $\lambda = a\varpi_1 + b\varpi_2 + c\varpi_3 \in X^+$ where $\{\varpi_1, \varpi_2, \varpi_3\}$ are the fundamental weights. The  three ``simple'' decompositions mentioned before come in very different flavors.

We use the notation $\alpha_{13}:=\alpha_1+\alpha_2+\alpha_3$, $\alpha_{12}:=\alpha_1+\alpha_2$ and $\alpha_{23}:=\alpha_2+\alpha_3$.
The first decomposition is rather simple.
\begin{equation} \label{eq n4 en n4 intro}
    \bfN_{\lambda}^4 = \displaystyle\sum_{k=0}^{\min(a, c)} q^k \bfN_{\lambda -k\alpha_{13}}^3.
\end{equation}
 This type of decomposition is a special case of a general phenomenon in type $A_n$: in fact, a similar behavior occurs in the decomposition of $\bfN^{i+1}$ in terms of $\bfN^{i}$ when $i\geq n/2+1$ (see \Cref{nhalf}).

For the second decomposition we need to introduce the set $I_{\lambda}$.
 It is the set of $\mu\in X^+$ such that there exist $n, m, l\in \mathbb{N}$  with  \[\mu=\lambda-n\alpha_{12}-m\alpha_{23} \ \ \ (\mathrm{in\ this\ case\ we\ consider\ }l\ \mathrm{to\ be\ }0)\] or \[\lambda-n\alpha_{12}-m\alpha_{23}\in \mathbb{N}\varpi_1+\mathbb{N}\varpi_3\ \mathrm{ and}\  \mu=\lambda-n\alpha_{12}-m\alpha_{23}-l\alpha_{13}\in X^+.\]   
For $\mu\in I_{\lambda}$ we define $d(\mu):=n+m+2l$. 
\begin{equation}\label{eq intro N3 in N2}
   \bfN_{\lambda}^3 = \displaystyle\sum_{\mu \in I_{\lambda}} q^{d(\mu)} \bfN_{\mu}^2.   
\end{equation}


The last decomposition is given by the formula 
\begin{equation}\label{1}
    \bfN_{\lambda}^2 = \displaystyle \sum_{\substack{\mu \in X^+ \\ \mu \leq \lambda \ \  } }  q^{\hgt(\lambda-\mu)} \bfN_{\mu}^1,
\end{equation}
where $\leq $ is the usual order on weights (i.e. $\mu\leq \lambda$ means that $\lambda-\mu\in \mathbb{N}\alpha_1+\mathbb{N}\alpha_2+\mathbb{N}\alpha_3)$ and $\hgt$ denotes the height of a weight.

The basis $\bfN_\lambda^2$ would be the canonical basis ``if all Kazhdan-Lusztig polynomials were trivial." Geometrically, $\bfN_\lambda^2$ is the character of the constant sheaf on the corresponding Schubert variety.
\Cref{1} will remain valid for any root system. 

\subsection{Pre-canonical bases.} Let us fix some notation before we can introduce the pre-canonical bases. 
Let $\Phi$ be a root system and $X\supset \Phi$ be a corresponding weight lattice. We fix a system of positive roots $\Phi^+\subset \Phi$ and let $\Delta$ be the corresponding set of simple roots. Let $\rho$ be the half-sum of the positive roots.  
 For an integer $ i\geq 1$ define $\Pgeq{i}$ to be the set of positive roots with height at least $i$, or in formulas
\[\Pgeq{i}:=\{ \alpha \in \Phi^+ : \mathrm{ht}(\alpha)\geq i   \}.\]

Let $W_f$ be the finite Weyl group attached to $\Phi$. We say that a weight $\lambda\in X$ is \emph{regular} if there is no reflection $s\in W_f$ which fixes $\lambda+\rho$.   For $\lambda\in X$ regular we define $w_\lambda\in W_f$ to be the unique element such that $w_\lambda \cdot \lambda $ is dominant (here $\cdot$ stands for the dot action, defined as $w\cdot \lambda=w(\lambda+\rho)-\rho$).

Let 
 $W_a=W_f\ltimes \bbZ \Phi$ be the corresponding affine Weyl group. 
Let $\calH$ be the Hecke algebra of $W_a$ over $\bbZ[q^{\frac12},q^{-\frac12}]$  with standard basis $\{\bfH_x\}$ and Kazhdan-Lusztig basis $\{\undH_x\}$.
For $\lambda \in \bbZ\Phi\subset X$, we think of the translation $t_\lambda$ as an element of $W_a$. 
 If $\lambda \in \bbZ\Phi\cap X^+$, let $\theta(\lambda)= t_\lambda w_0$ and $\undH_\lambda:=\undH_{\theta(\lambda)}\in \calH$.
 Then, for $\lambda \in \bbZ\Phi$, we  define
	\begin{equation*}
	\widetilde{\undH}_\lambda	=  \left\{
	\begin{array}{rl}
	(-1)^{\ell(w_{\lambda})} \undH_{w_{\lambda}\cdot \lambda},	& \mbox{if } \lambda \mbox{ is regular;}   \\
	0,& \mbox{if } \lambda \mbox{ is singular.}
	\end{array} 
	\right.
	\end{equation*}
(See \Cref{dw} for how to extend the definition of $\HTIL_\lambda$ to any $\lambda\in X$.)

\begin{definition}[Pre-canonical bases]
For  $i\geq 2$ and $\lambda\in X^+$ define
\begin{equation}\label{pc} 
 	\bfN_{\lambda}^{ i} :=\sum_{I\subset \Pgeq{i} } (-q)^{|I|}\widetilde{\undH}_{\lambda - \sum_{\alpha\in I}\alpha}.
 \end{equation}

\end{definition}

For $i=1$ the definition is almost the same as \cref{pc}, only that one has to normalize by some scalar (for details see Definition \ref{First}).
For any fixed $i\geq 1$, the set $\bfN^i:=\{\bfN_{\lambda}^{i} : \lambda\in  X^+\}$ is called the $i^{\text{th}}$ \emph{pre-canonical basis}. It is a basis of the \emph{spherical Hecke algebra} $\Htil $,  the decategorification of any of the two categories appearing in the Geometric Satake equivalence (see \cref{omega} for an easy combinatorial definition of $\Htil$ and \cref{standard}  for the definition of its standard basis  $\{\bfH_{\lambda}\}$). 
 The second part of the following theorem is a $q$-deformation of the main  result of Sch\"{u}tzer's paper  \cite{Sch} and it is at the origin of the definition of the pre-canonical bases. 

\begin{theorem}\label{12} For $\lambda\in X^+$ we have the following equations
\begin{equation*}
    \bfN_{\lambda}^{1} =\bfH_{\lambda}  \qquad \mbox{and} \qquad
    \bfN_{\lambda}^2   = \displaystyle \sum_{\substack{\mu \in X^+ \\ \mu \leq \lambda } }  q^{\hgt(\lambda-\mu)} \bfN_{\mu}^1.
\end{equation*}
\end{theorem}
 The fact that $\bfN_{\lambda}^{m+1}=\underline{\bfH}_{\lambda}$ if $m$ is the height of the highest root follows directly the definition of the pre-canonical bases. We remark that equality $\bfN_\lambda^1=\bfH_\lambda$ gives a closed formula for all inverse (spherical) Kazhdan-Lusztig polynomials for all affine Weyl groups. 


\begin{remark}
The spherical Hecke algebra is isomorphic via the \emph{Satake transform} to the algebra of symmetric functions $\bbZ[v,v^{-1}][X]^{W_f}$ (cf. \Cref{geoSat}). Under this isomorphism, the Kazhdan-Lusztig basis $\{\undH_{\lambda}\}$ corresponds to the Weyl characters and the standard basis $\{\bfH_{\lambda}\}$ to the Hall-Littlewood polynomials (see e.g. \cite{Ste3}). It follows that, after applying Satake, the pre-canonical bases also yield new bases of the ring of symmetric functions which interpolate between Weyl characters and Hall-Littlewood polynomials.
\end{remark}

\subsection{Main Conjecture.} Suppose now that $\Phi$ is a root system of  type $A_n$. For an integer $1\leq i \leq n$, let $\Phi^{i}$ be the set of positive roots of height  $i$.
 For $\lambda , \mu \in X^+$, we write $\mu \leq_i \lambda$ if $\lambda - \mu$ can be written as a positive integral linear combination of elements of $\Phi^i$.  
 
\begin{theorem} \label{nhalf}
    Let $n/2+1 \leq  i \leq n $.  For  $\lambda \in X^+$ we have 
    \begin{equation*}
        \bfN_{\lambda}^{i+1} =  \displaystyle \sum_{\mu \leq_i \lambda } q^{\frac{1}{i} \hgt (\lambda - \mu)} \bfN_{\mu}^i.
    \end{equation*}
     
\end{theorem}
The formula in \Cref{nhalf} is incredibly simple, and although we hope that a formula for all $i$ will be found in the future (cf. \Cref{empirical}), we do not expect it to be as simple as that. In fact, we expect a different behavior for small $i$ much more in the vein of \cref{1}, as the following theorem illustrates.

\begin{theorem} \label{A3A4}
    The formulas for all the decompositions of $\bfN^{i+1}$ in terms of $\bfN^{i}$ in type $A_3$  are the ones explained in Section \ref{a3}. The corresponding formulas  in type $A_4$ are the ones  explained in Section \ref{a4}. 
\end{theorem}
We remark that the corresponding formulas in type $A_2$ were found by the first two authors of this paper in \cite{libedinsky2020affine} and in type $A_1$ they are trivial. The following is the central conjecture of this paper. 
 \begin{conj}\label{positivity}
 If $\Phi$ is a root system of  type $A_n$, for each $i\geq 0$  we have
 \[ \bfN_{\lambda}^{ i+1} \in \sum_{\mu \in X^+}\mathbb{N}[q]\ \bfN_{\mu}^{ i} \]
 \end{conj}
 
 \begin{remark}
 The only non trivial part of the conjecture is the positivity statement. We have the following evidence to believe in the validity of this conjecture. By \Cref{12} the conjecture is  verified for $i=1$ and all $n$. By \Cref{nhalf} it is also verified for $n/2+1 \leq  i$. By \Cref{A3A4} the conjecture is verified for all $i$ when $n\leq 4$.  It was proved by Shimozono \cite{Shimo} following earlier work by Lascoux \cite{Lasc} that in type $A_n$ the decomposition of $\bfN_{\lambda}^{n+1}$ in terms of the $\bfN^2$ basis (that using \Cref{12} one proves that it is the so-called \emph{atomic decomposition}) is positive. Finally, we have checked this conjecture in several hundred cases by computer in types $A_5$ and $A_6$ using SageMath \cite{sagemath} with the help of the code developed in \cite{HIS}.
 \end{remark}
 
\begin{remark}\label{empirical}
Conjecture \ref{positivity}, as written, is false for a general root system. In type $D_4$ a counter-example \cite[Example 2.6]{LeLe} shows that $\underline{\bfH}_\lambda$ is not  positive in general when decomposed in terms of $\bfN^2$. However, as Lecouvey and Lenart explain, the failure of this positivity seems to be mild and it would be interesting to determine whether there exists a ``stable range'' where the conjecture holds.
\end{remark} 
 
 \begin{remark}
The name ``pre-canonical bases'' is inspired by this conjecture. We see Lusztig generational philosophy as a conjectural set of ``post-canonical bases'', this time interpolating between the canonical basis and the $p$-canonical basis.
\end{remark} 
 
 \begin{remark}
 Empirical data in ranks $\leq 8$ suggest that it should be possible to find a combinatorial formula for the polynomials appearing in the right-hand side of Conjecture \ref{positivity}. More precisely, we believe that for $i>1$ the polynomials $P_i(\lambda, \mu)\in \bbN[q]$ defined by the formula
 \[ \bfN_{\lambda}^{ i+1} = \sum_{\mu \in X^+}P_i(\lambda, \mu)\ \bfN_{\mu}^{ i}, \] 
 can be computed by constructing a subset 
 $\mathfrak{L} \subseteq \mathcal{L}_{\lambda}^i(\mu)$ so that 
 \begin{equation*}
     P_i(\lambda, \mu) = \sum_{L\in \mathfrak{L}} q^{\deg_i (L)}.
 \end{equation*}
 Here $\mathcal{L}^i_{\lambda}(\mu)$ is the set of all non-negative linear combinations of elements in  $\Phi^i \cup \Phi^{2i-1} \cup \Phi^{3i-2} \cup \ldots$ (i.e., roots $\alpha\in \Phi^+$ such that $i-1$ divides $\hgt(\alpha)-1)$) equal to $\lambda -\mu$  and  $\deg_i $ is a function defined on any positive root by the formula 
 \[\deg_i(\alpha):=\frac{\hgt(\alpha)-1}{i-1}.\]
  and extended $\bbZ$-linearly to any element in $\mathcal{L}^i_{\lambda}(\mu)$.

  This idea is reminiscent of Deodhar's proposal \cite{deodhar1990combinatorial} for a counting formula for Kazhdan-Lusztig polynomials recently refined by the first author and Geordie Williamson \cite{LW2}.
 \end{remark}

\subsection{Structure of the paper.} The paper  is structured as follows. We start in \Cref{dw} by reviewing root systems and (extended) affine Weyl groups (in \Cref{exampleAn} we focus on type A and give a more elementary description of these objects).
 In \Cref{satake} 
we define the spherical Hecke algebra together with its standard and Kazhdan-Lusztig bases and in \Cref{prec} we define the pre-canonical bases in detail and prove that they are bases of the spherical Hecke algebra. 

In \Cref{antiatomicsec} we prove \Cref{12} (i.e. give a formula for inverse Kazhdan-Lusztig polynomials and for inverse atomic polynomials) using root system combinatorics. In \Cref{upp} we prove \Cref{nhalf} by  defining several $\bfM^j_{\lambda}$ that interpolate between $\bfN^i_{\lambda}$ and $\bfN^{i+1}_{\lambda}$ and behave particularly well under the Weyl group action. Finally in \Cref{last} we prove \Cref{A3A4} by first defining some bases $\hat{\bfN}^{i+1}$, then analyzing  the decomposition of $\bfN^i$ in $\bfN^{i+1}$ and then checking that this decomposition is the same as the one of $\bfN^i$ in $\hat{\bfN}^{i+1}$.

\subsection{Acknowledgments.} We would like to thank the referee for pointing out the connection of this work with Hall-Littlewood polynomials and several other points that improved the quality of the exposition.

\section{Definition of the pre-canonical bases}\label{S2}

 In this section we introduce the (extended) affine Weyl group and the corresponding Hecke algebras. We refer to \cite{W2} and \cite{Knop} for more details.
 
\subsection{Dominant weights and affine Weyl groups.}\label{dw}

The reader mostly interested in type $A_{n}$   can skip this section  (that is slightly technical) and read directly \Cref{exampleAn} instead. We will need the definitions given in this section to define the spherical Hecke algebra (where the pre-canonical bases live) in Section \ref{satake}. Our definition is equivalent to that given on the literature (cf. \Cref{knop}) but it is more natural from the point of view of categorification (cf. \Cref{geoSat}).

Let $(X,\Phi,X^\vee, \Phi^\vee)$ be a reduced root datum where $X$ is the character lattice  with roots $\Phi$, and $X^\vee$ is the  cocharacter lattice with coroots $\Phi^\vee$. We fix a system of simple roots $\Delta$ and positive roots $\Phi^+$. 
We assume that our root datum is simply connected, i.e. that $X^\vee = \bbZ \Phi^\vee$. Let $\rho$ be the half-sum of all the positive roots, and $\rho^\vee$ the half-sum of all the positive coroots.
Let $\langle-,-\rangle$ denote the pairing between weights $X$ and coweights $X^\vee$.

We denote by $\leq$ the dominance order on $X$: we say that $\lambda\leq \mu$ if $\mu-\lambda$  can be written as a integral linear combination of elements in $\Phi^+$.

Let $X_\bbR:= X\otimes_\bbZ \bbR$. For a root $\alpha \in \Phi$, let $s_\alpha$ denote the corresponding reflection:
$s_\alpha:X_\bbR\ra X_\bbR$ defined as $
s_\alpha(v)=v-\langle v,\alpha^\vee\rangle \alpha$.
Let $S_f$ be the set of reflections $s_\alpha$, with $\alpha \in \Delta$. The group $W_f$ generated by $S_f$ is the (finite) Weyl group. We denote by $w_0$ the longest element in $W_f$.

The affine Weyl group $W_a$ is the subgroup of affine transformations of $X_\bbR$ generated by $W_f$ and $\bbZ\Phi$ (acting as translations). We have $W_a \cong  W_f\ltimes \bbZ \Phi$. 
The group $W_a$ can also be described as the group generated by $s_{\alpha,m}$, the reflections along the hyperplanes 
\[H_{\alpha,m}=\{\lambda \in X_\bbR \mid \langle \lambda,\alpha^\vee \rangle =m \},\]
for $\alpha\in \Phi$ and $m\in \bbZ$. The connected components of the complement of the hyperplanes \[X_\bbR \setminus \bigcup_{\alpha,m}H_{\alpha,m}\] are called alcoves.
We call $C_0=\{\lambda \in X_\bbR \mid -1 < \langle \lambda,\alpha^\vee\rangle < 0$ for any $\alpha\in \Phi^+\}$ the fundamental alcove. Then the map $w\mapsto wC_0$ defines a bijection between $W_a$ and the set of alcoves.

The walls of $C_0$ are $H_{\alpha,0}$, for $\alpha \in \Delta$, and $H_{\beta,-1}$ for $\beta$ the longest short root (so that $\beta^\vee$ is the longest coroot). We set $s_0:=s_{\beta,-1}$ to be the reflection along $H_{\beta,-1}$. Then $W_a$ is a Coxeter group with simple reflections $S=S_f \cup \{ s_0\}$.

We also consider the extended affine Weyl group $W_e$: this is the subgroup of affine transformations of $X_\bbR$ generated by $W_f$ and $X$ (where $X$ acts as translations). We have $W_{e} = W_f\ltimes  X$.\footnote{In the literature (e.g \cite{Bou,Knop}) the extended affine Weyl group is often defined as the group generated by $W_f$ and the coweight lattice $X^\vee$.}
Although $W_{e}$ is not a Coxeter group in general, one can still define the length $\ell(w)$ of an element $w\in W_e$ by counting how many hyperplanes separate $C_0$ and $wC_0$.

Let $\Omega$ be the subgroup of length $0$ elements in $W_e$. 
In other words, this is the subgroup of elements $\sigma\in W_e$ such that $\sigma (C_0)=C_0$. Hence, every element of $\Omega$ permutes the walls of $C_0$, therefore conjugation by $\Omega$ permutes the simple reflections in $W_a$, so $\Omega$
can be seen as a group of automorphisms of the Dynkin diagram of $W_a$.

The group  $\Omega$ is isomorphic to the fundamental group 
$X/\bbZ \Phi$ of the root datum. In fact, for $\lambda \in X$, let $t_\lambda\in W_e$ denote the corresponding translation. Then $w_0(C_0)+\lambda$  is an alcove, so there exists a unique element $\theta(\lambda)\in W_a$ such that $\theta(\lambda)(C_0)=w_0(C_0)+\lambda$. The map $\lambda\mapsto \theta(\lambda)^{-1}t_\lambda w_0$ defines a surjective group homomorphism from $X$ to $\Omega$ with kernel $\bbZ \Phi$. It follows that $W_e/W_a\cong X/\bbZ\Phi\cong \Omega$ and $W_{e}\cong W_a\rtimes \Omega$.

Moreover, we have $W_e/W_f\cong X$, so $W_f\backslash W_e/W_f\cong X^+$, where $X^+$ is the set of dominant weights.
We can actually refine this bijection. In fact, we have compatible decompositions
\begin{equation}\label{x+}
\begin{tikzcd}
X^+ \arrow{r} \arrow[r,"\sim"]\isoarrow{d}&  \bigsqcup_{\sigma \in X/\bbZ \Phi} (\sigma +\bbZ \Phi) \cap X^+ \isoarrow{d}\\
W_f\backslash W_e/W_f  \arrow[r,"\sim"] & \bigsqcup_{\sigma\in \Omega} W_f\backslash W_a/\sigma(W_f)
\end{tikzcd}
\end{equation}
where $\sigma(W_f):=\sigma W_f \sigma^{-1}$. 
The bottom horizontal arrow sends the double coset containing $t_\lambda$, $\lambda\in X^+$, to the double coset containing $\theta(\lambda)$, for $\sigma=\theta(\lambda)^{-1}t_\lambda w_0$.
In particular, for every $\sigma \in \Omega$, the map $\theta$ defines a bijection
\begin{equation*}\label{weighttocosets}\theta: (\sigma +\bbZ \Phi) \cap X^+ \xrightarrow{\sim} W_f\backslash W_a/\sigma(W_f).
\end{equation*}
This bijection intertwines the Bruhat order on the right with the dominance order on the left. Moreover, we have \[\ell(\theta(\lambda))=\ell(w_0)+2 \langle \lambda,\rho^\vee\rangle\] 
(see for example \cite[Eq. (2.9)]{Knop}).

\begin{ex}\label{exampleAn}

Since an important part of this paper is devoted exclusively to type $A_n$, we spell out the definitions in the previous section in this case.

Assume that $n\geq 2$.
 Let $X_\bbR\subset \bbR^{n+1}$ be the subspace of vectors with coordinates adding up to zero. The set of roots is $\Phi:=\{\epsilon_i-\epsilon_j : 1\leq i\neq j\leq n+1\}$. The simple roots are \[\Delta:=\{\alpha_1:=\epsilon_1-\epsilon_2,\;\alpha_2:=\epsilon_2-\epsilon_3,\;\ldots\;,\; \alpha_{n}:=\epsilon_{n}-\epsilon_{n+1}\},\] and the fundamental weights are 
 \[ \varpi_i:=(\epsilon_1+\cdots+\epsilon_i)-\frac{i}{n+1}\sum_{j=1}^{n+1}\epsilon_j. \]
 The character lattice (or weight lattice) is $X:=\mathbb{Z}\varpi_1+\mathbb{Z}\varpi_2+\cdots+\mathbb{Z}\varpi_{n}$ and the set of dominant weights is  $X^+:=\mathbb{N}\varpi_1+\mathbb{N}\varpi_2+\cdots+\mathbb{N}\varpi_{n}$.
  For $1\leq i< n+1$, define $s_i$ to be the reflection in $X_\bbR$ that permutes $\epsilon_i$ and $\epsilon_{i+1}$ (in other words, the reflection that fixes the hyperplane $x_i+x_{i+1}=0$.) The Weyl group   $W_f\subset \mathrm{End}(X_\bbR)$ is the subgroup of  endomorphisms of the vector space $X_\bbR$ generated by the $s_i$ with $1\leq i< n+1$. Let $s_0$ be the reflection through the hyperplane with equation $x_1+x_{n+1}+1=0$. The affine Weyl group $W_a$ is the subgroup  of affine endomorphisms of $X_\bbR$ generated by $S:=\{s_i : 0\leq i\leq  n\}$.

 The Dynkin diagram of $W_a$ can be seen as a regular $(n+1)$-gon with vertices $S$. Let $\Omega\cong \mathbb{Z}/(n+1)\mathbb{Z}$ be the group of cyclic permutations of this diagram (see also \Cref{Omega}). The element $\sigma_i\in \Omega$ for $1\leq i\leq n+1$ is defined by $\sigma_i(s_j)=s_{i+j}$ (where the sub-index $i+j$  is understood modulo $n+1$).

Let $w_0$ be the longest element in $W_f$.
 We now define a function $\theta: X^+\rightarrow W_a$ as follows. If $\lambda\in X^+\cap \bbZ \Phi$, then the translation $t_\lambda$ is an element of $W_a$ and we define $\theta(\lambda):=t_\lambda w_0$. If $\lambda\in X^+\setminus \bbZ \Phi$, then there exists $1\leq i\leq n$ such that $\lambda=\mu +\varpi_i$ and $\mu\in \bbZ\Phi \cap X^+$.  Then we define 
 $\theta(\lambda)=\theta(\mu)w_0 \theta(\varpi_i)$, where 
 $\theta(\varpi_i)$ is the longest element in $W_f\sigma_i(W_f)$.
 More explicitly, we have
 \[w_0\theta(\varpi_i)=
   \prod_{k=1}^i s_{[1-k,n-i-k+1]}.\]
where, for $a\leq b$, we define $s_{[a,b]}:=s_{a}s_{a+1}\cdots s_{b}$ (again, the sub-indices are understood modulo $n+1$).
By construction, we have that the right descent set of $\theta(\lambda)$ is $S\setminus\{s_i\}$.

An important property of this map is that $\lambda\leq \mu$ in the dominance order if and only if $\theta(\lambda)\leq \theta(\mu)$ in the (strong) Bruhat order. 
\end{ex}

\begin{remark}\label{Omega}
If $\Phi$ is of type $A_n$, as mentioned in \Cref{exampleAn}, 
the group $\Omega$ of length $0$ element in $W_e$ is isomorphic to the group $\bbZ/(n+1)\bbZ$ which acts on the Dynkin diagram by cycling the simple reflections $\{s_0,s_1,\ldots s_n\}$ of $W_a$.
We explain here how to obtain this isomorphism.

The isomorphism $\Omega \ra \bbZ/(n+1)\bbZ$ is given by the map $\sigma \mapsto j$, where $j$ is the index of the simple reflection $\sigma(s_0)=\sigma s_0\sigma^{-1}\in S$. For any $0\leq j\leq n$, we can define $\sigma_j\in \Omega$, to be the unique element which sends $s_0$ to $s_j$.
Then, an element $x\in W_a$ is a  maximal element in its double coset $W_f x\sigma_j(W_f)\in  W_f\backslash W_a/ \sigma_j(W_f)$ if and only if its left descent set is $S\setminus\{s_0\}$ and its right descent set is $S\setminus\{s_j\}$.

On the other hand, the fundamental weights $\varpi_i$, $1\leq i \leq n$, together with $0$ form a set of representatives of $X/\bbZ \Phi$. So we can canonically identify $X/\bbZ\Phi$ with $\bbZ/(n+1)\bbZ$ by sending $\varpi_i$ to $i\in \bbZ/(n+1)\bbZ$ (see \cite[Prop VI.2.3.6]{Bou} for more details).
Hence, if we write a weight $\lambda$ in the basis of fundamental weights as $\lambda= \sum_{i=0}^n a_i\varpi_i$, then its class in $\bbZ/(n+1)\bbZ$ is given by $\sum_{i=1}^n ia_i$. We can summarize this information in the following commutative diagram.

\begin{center}
\begin{tikzpicture}
\node (A) at (0,0) {$X/\bbZ\Phi$};
 \node (B) at (4,0) {$\Omega$};
 \node (C) at (2,-2) {$\bbZ/(n+1)\bbZ$};
\draw[->] (A) -- node[left]{$\varpi_i\mapsto  i$} (C);
\draw[->] (A) --node[above]{$\lambda\mapsto \theta(\lambda)^{-1}t_\lambda w_0$} (B); 
\draw[->] (B) -- node[right]{$\sigma \mapsto j$ s.t. $\sigma(s_0)=s_j$} (C);
\end{tikzpicture}
\end{center}

Every fundamental weight $\varpi_i$ is the minimal element in the set $(\sigma +\bbZ \Phi) \cap X^+$, hence $\theta(\varpi_i)$ is the longest element in the double coset $W_f \sigma_i(W_f)$.
More generally, we can compute $\theta$ recursively as follows. Assume that we know $\theta(\lambda)$ for $\lambda\in X^+$ and let $\sigma\in \Omega$ be the class of $\lambda$, then $\theta(\lambda+\varpi_i)=\theta(\lambda)\sigma(w_0^{-1}\theta(\varpi_i))=\theta(\lambda)\sigma w_0^{-1}\theta(\varpi_i)\sigma^{-1}=t_\lambda \theta(\varpi_1)\sigma^{-1}$.
Notice that $\ell(w_0^{-1}\theta(\varpi_i))=2\langle \varpi_i,\rho^\vee\rangle=\ell(\theta(\lambda+\varpi_i))-\theta(\lambda)$, so we obtain a reduced expression of $\theta(\lambda+\varpi_i)$ simply by stacking together a reduced expression of $\theta(\lambda)$ and one of $\sigma(w_0^{-1}\theta(\varpi_i))$. By recursion, in this way one can easily obtain  reduced expressions of every $\theta(\lambda)$, for $\lambda\in X^+$.
\end{remark}

\subsection{The spherical Hecke Algebra.}\label{satake}


Let $\calH$ be the Hecke algebra of $W_a$ over $\bbZ[v,v^{-1}]$ with standard basis $\{\bfH_x\}$ and Kazhdan-Lusztig basis $\{\undH_x\}$. Recall that we use the convention $q:=v^2$.
The action of $\Omega$ on $W_a$ naturally extends to an action by algebra automorphisms on $\calH$, where $\sigma\in \Omega$ sends $\bfH_x$ to $\bfH_{\sigma(x)}$ for any $x\in W_a$.

Let $\undH_f:=\undH_{w_0}$ be the Kazhdan-Lusztig basis element for the longest element $w_0\in W_f$. We define
\[\calH^\sigma:=\undH_f \calH \cap \calH \sigma(\undH_f)\subset \calH\]
Similarly, for $\sigma, \tau \in \Omega$ we can define 
${}^\sigma\calH^\tau:=\sigma(\undH_f) \calH \cap \calH \tau(\undH_f)$. Notice that the action by $\tau$ induces an isomorphism of $\bbZ[v,v^{-1}]$-modules
$\calH^\sigma \cong {}^\tau\calH^{\tau\sigma}$.
 
 For a finite subgroup $H$ of $W_a$ we denote by $\pi_{H}(q)$ the \emph{Poincar\'e polynomial} of $H$, defined as $\pi_{H}(q)=\sum_{w\in H}q^{\ell(w)}$.
As in \cite[\S 2.3]{W4}, we can arrange all the $\calH^\sigma$, for $\sigma \in \Omega$, in an $\Omega$-graded algebra 
\begin{equation}\label{omega}
\Htil:= \bigoplus_{\sigma \in \Omega} \calH^\sigma
\end{equation}
where multiplication
$\calH^\tau \times \calH^\sigma \ra \calH^{\tau\sigma}$ is defined via
\[\calH^\tau \times \calH^\sigma \xrightarrow{\sim} \calH^\tau \times {}^\tau\calH^{\tau\sigma}  \ra \calH^{\tau\sigma}\]
\[ (x,y) \ra (x,\tau(y)) \ra \frac{v^{\ell(w_0)}}{\pi_{W_f}(v^2)}x\tau(y)\]
(here $x\tau(y)$ is the product of $x$ and $\tau(y)$ as elements of $\calH$).
Notice that $\undH_f$ is the unity of the algebra $\Htil$ (cf. \cite[Eq. (2.2.4)]{W4}). For $\lambda\in X^+$ we define 
\begin{equation}\label{standard}
\bfH_\lambda:=\sum_{x\in W_f\theta(\lambda)\sigma(W_f)}v^{\ell(\theta(\lambda))-\ell(x)}\bfH_x
\end{equation}
Then $\{\bfH_\lambda\}$ is the standard basis of $\Htil$.
For an element $x\in W_a$, we have $\undH_x\in \calH^\sigma$ if and only if $x$ its maximal in its double coset $W_f x \sigma(W_f)$. Since $\theta(\lambda)$ is by definition maximal in its double coset, we can define $\undH_\lambda:=\undH_{\theta(\lambda)}\in \calH^\sigma$. The set $\{\undH_\lambda\}$ is the Kazhdan-Lusztig basis of $\Htil$.

We can write 
\[\undH_\lambda = \sum_{\mu \leq \lambda} h_{\mu,\lambda}(v)\bfH_\mu\]
where $h_{\lambda,\lambda}(v)=1$ and $h_{\mu,\lambda}(v)\in v\bbZ_{\geq 0}[v]$.  We set $h_{\mu,\lambda}(v)=0$ if $\mu\not\leq \lambda$. The polynomials $h_{
\mu,\lambda}(v)$ are called Kazhdan-Lusztig polynomials, and coincide with the Kazhdan-Lusztig polynomials of $\calH$, i.e. we have  $h_{\mu,
\lambda}(v)=h_{\theta(\mu),\theta(\lambda)}(v)$.
\begin{remark}\label{knop}
The definition of the spherical Hecke algebra used in \cite{Knop} is slightly different. Knop defines it as a subring of the \emph{extended affine Hecke algebra}. It is easy to use the bijections in \cref{x+} to show that the two definitions are equivalent.
\end{remark}

\begin{remark}\label{geoSat}
There is an isomorphism $\Htil \cong \bbZ[v,v^{-1}][X]^{W_f}$, called the Satake isomorphism (see \cite{Knop,Ste3} for more details). Under this isomorphism, the Kazhdan-Lusztig basis $\undH_\lambda$ is sent to the Weyl characters and the standard basis $\bfH_\lambda$ is sent to Hall-Littlewood polynomials of the reductive group $G$ associated with the root datum \cite{Lu3}. 

The Satake isomorphism admits a categorification, called the \emph{geometric Satake isomorphism} \cite{MV}, as an equivalence of monoidal categories between the category of finite dimensional representations of $G$ and the category of equivariant perverse sheaves on the affine Grassmannian of the Langlands dual group $G^\vee$. 
\end{remark}

We can define an additional basis of $\Htil$ as follows.
\begin{equation}\label{Nbasis}\bfN_\lambda := \sum_{\mu \leq \lambda} v^{2\hgt(\lambda-\mu)}\bfH_\mu =\sum_{\mu \leq \lambda} v^{\ell(\theta(\lambda))-\ell(\theta(\mu))}\bfH_\mu. 
\end{equation}

For $x\in W_a$, define 
\begin{equation}\label{N}
    \bfN_x:=\sum_{y\leq x}v^{\ell(x)-\ell(y)}\bfH_y\in \mathcal{H}.
\end{equation} 

It is easy to check that equations \cref{Nbasis} and \cref{N} coincide, i.e. $\bfN_{\lambda}=\bfN_{\theta(\lambda)}$.
We can write
\[ \undH_{\lambda} = \sum_{\mu \leq \lambda} a_{\lambda,\mu}(v) \bfN_\mu\]
for some polynomials  $a_{\lambda,\mu}(v)\in \bbZ[v]$. The polynomials $a_{\lambda,\mu}(v)$ are called \emph{atomic polynomials} and one says that $\undH_\lambda$ admits an \emph{atomic decomposition} if $a_{\lambda,\mu}(v)\in \bbZ_{\geq 0}[v]$ for every $\mu\in X^+$. We remark that we follow the convention from \cite{LeLe}, but the sub-indices are inverted in the $a$ polynomial with respect to Kazhdan-Lusztig polynomials, i.e. in the first position one has the biggest weight.  

In type $\tilde{A}_n$ there is an atomic decomposition, as proved by Lascoux \cite{Lasc} and Shimozono \cite{Shimo}. In \cite{LeLe} the atomic decomposition in type $\tilde{A}_n$ was reproved using crystals. In their paper, Lecouvey and Lenart also consider  other types and, even if the atomic decomposition does in general fail, they conjecture that it still holds under some mild assumptions on $\lambda$.

\subsection{Kostka-Foulkes polynomials.}

Kazhdan-Lusztig polynomials in the spherical Hecke algebra can be reinterpreted, up to normalization, as Kostka-Foulkes polynomials $K_{\lambda,\mu}(q)$, and in particular they give $q$-weight multiplicities of the irreducible representations of the reductive group associated to the root datum \cite{Lu3}.
To recall the definition of the Kostka-Foulkes polynomials we first need to introduce the ($q$-analog of) the Kostant partition function.

\begin{definition}
Let  $Q^+\subset X$ be the positive part of the root lattice, i.e. $Q^+$ is the subset of weights that can be written as $\sum_{i=1}^n c_i\alpha_i$, with $c_i\in \bbZ_{\geq 0}$. Let $\kpf_q:Q^+\ra \bbQ[q]$ be the function defined by 
\begin{equation}\label{kpf}
\prod_{\alpha\in \Phi^+}\left(\sum_{k\geq 0} q^k e^{-k\alpha}\right)=\sum_{\alpha\in Q^+} \kpf_q(\alpha)e^{-\alpha}. 
\end{equation}

The function $\kpf_q:Q^+\ra \bbQ[q]$ is called the \emph{$q$-analog of the Kostant partition function}.
We trivially extend the definition of $\kpf_q$ to the set of weights $X$ by imposing that $\kpf_q(\alpha)=0$ if $\alpha\not \in Q^+$.
 \end{definition}

\begin{definition}
For $\lambda,\mu\in X^+$ with $\mu\leq \lambda$, the \emph{Kostka-Foulkes polynomials} are defined as 
\begin{equation}\label{KF} K_{\lambda,\mu}(q)=\sum_{w\in W_f}(-1)^{\ell(w)} \kpf_q(w(\lambda+\rho)-\mu -\rho).
\end{equation}
\end{definition}

For any $\lambda,\mu\in X^+$ with $\mu\leq \lambda$ we have $K_{\lambda,\mu}(q)\in \bbZ_{\geq 0}[q]$. 
We remark that in general \Cref{KF} does not lead to a polynomial with positive coefficients if $\mu\not\in X^+$.

Kostka-Foulkes polynomials are the Kazhdan-Lusztig polynomials in the spherical Hecke algebra. More precisely, we have by \cite[Theorem 1.8]{kato1982spherical} 
\[h_{\mu,\lambda}(v)=K_{\lambda,\mu}(v^2).\]
(Kato uses the alternative parametrization $h_{\mu,\lambda}(v)=v^{\ell(\theta(\lambda))-\ell(\theta(\mu))}P_{\theta(\mu),\theta(\lambda)}(v^{-2})$ of Kazhdan-Lusztig polynomials.)  We remark again that we follow the conventions in the literature and that the Kostka-Foulkes polynomials have (as the atomic polynomials) their sub-indices inverted with respect to the Kazhdan-Lusztig polynomials.

\subsection{The pre-canonical bases.}\label{prec}
We introduce now a new set of bases of the spherical Hecke algebra $\Htil$, that we call the pre-canonical bases. Roughly speaking, one can think of these new bases as an interpolation between the standard basis and the Kazhdan-Lusztig basis.

\begin{definition}
We say that a weight $\lambda\in X$ is \emph{singular} if there is an element $s\in W_f$ which fixes $\lambda+\rho$. Equivalently, $\lambda$ is singular if there is a root $\alpha$ such that $\langle \lambda + \rho,\alpha^\vee\rangle = 0$. A non-singular weight is called \emph{regular}.
\end{definition}
Let us recall that the \emph{dot action} (also called the \emph{affine action}) of the finite Weyl group $W_f$ on the set of weights is given  by the formula
\[w\cdot \lambda= w(\lambda +\rho)-\rho.\]
Notice that every dominant weight is regular. In the other direction, for $\lambda\in X$ regular we define $w_\lambda\in W_f$ to be the unique element such that $w_\lambda \cdot \lambda $ is dominant. Moreover, let $\overline{\lambda}:=w_\lambda\cdot \lambda \in X^+$.
We extend the definition of $\undH_\lambda$ to non-dominant weights.

\begin{definition} \label{defi H tilde bar}
	We define
	\begin{equation*}
	\widetilde{\undH}_\lambda	=  \left\{
	\begin{array}{rl}
	(-1)^{\ell(w_{\lambda})} \undH_{\overline {\lambda}},	& \mbox{if } \lambda \mbox{ is regular;}   \\
	0,& \mbox{if } \lambda \mbox{ is singular.}
	\end{array} 
	\right.
	\end{equation*}
\end{definition}

For an integer $ i\geq 1$ define $\Pgeq{i}$ to be the set of positive roots with height at least $i$.
We are now ready to  define the main characters of the present paper. 

\begin{definition}[Pre-canonical bases]\label{definition pre-canonical basis}
 For  $i\geq 2$ and $\lambda\in X^+$ we define
\begin{equation} \label{defin elements N}
 	\bfN_{\lambda}^{ i} :=\sum_{I\subset \Pgeq{i} } (-v^2)^{|I|}\widetilde{\undH}_{\lambda - \sum_{\alpha\in I}\alpha}.
 \end{equation}
\end{definition}

As it will follow from Part \ref{part3} of \Cref{lambda-I}, for any $i\geq 2$, the set $\{\bfN^i_\lambda\}$
is unitriangular with respect to the Kazhdan-Lusztig basis $\{\undH_\lambda\}$. Therefore,  the set $\{\bfN^i_\lambda\}$ is indeed a basis of the spherical Hecke algebra $\Htil$.


\begin{definition}\label{phiw}
For $w\in W_f$, let $\Phi^+_w$ be the subset of positive roots $\alpha$ such that $w(\alpha) \in \Phi^+$ and let   $\Phi^+_{-w}=\Phi^+\setminus \Phi^+_w$. 
\end{definition}

For a subset $I\subset \Phi$ let $\Sigma_I:=\sum_{\alpha\in I}\alpha$. 

\begin{lemma}\label{lambda-I} 
Let $\lambda\in X^+$ and $I\subset \Phi^+$. Then
\begin{enumerate}
    \item $\overline{\lambda-\Sigma_I}\leq \lambda$.
    \item \label{part2}$\overline{\lambda-\Sigma_I}= \lambda$ if and only if there exists $w\in W_f$ such that $I=\Phi^+_{-w}$ and $w(\lambda)=\lambda$.
    \item \label{part3} For $i\geq 2$, if $I\subset \Pgeq{i}$ and $I\neq \emptyset$, then $\overline{\lambda-\Sigma_I}< \lambda$.
\end{enumerate}
\end{lemma}
\begin{proof}
This is proved in \cite[Theorem 2.1]{Sch}. We rewrite here the proof for convenience.

We have $\overline{\lambda-\Sigma_I}=w(\lambda+\rho-\Sigma_I)-\rho$ for $w=w_{\lambda-\Sigma_I}$. Since $w(\lambda)\leq \lambda$ and $w(\rho-\Sigma_I)\leq \rho$ by \cite[Lemma 5.9]{Kost}, we have $\overline{\lambda - \Sigma_I}\leq \lambda$. Moreover, $w(\rho-\Sigma_I)= \rho$ if and only if $I=\Phi^+_{-w}$ by \cite[Lemma 4.8]{Sch}. Hence $\overline{\lambda - \Sigma_I}= \lambda$ if and only if there exists an element
 $w\in W_f$ such that $I=\Phi^+_{-w}$  and $w(\lambda)=\lambda$. This shows the first two parts of the Lemma.
 
 For the last part, notice that for $w\neq id$, the set $\Phi^+_{-w}$ always contains a simple root, so it cannot be contained in $\Pgeq{i}$ for $i\geq 2$.
\end{proof}

 We would like to define $\bfN^1_\lambda$ similarly to \cref{defin elements N}, but we need to take into account that Part \ref{part3} of \Cref{lambda-I} does not hold for $i=1$ since there exists subsets $I\subset \Pgeq{1}$ with $I=\Phi^+_{-w}$ and $w\in W^f$. This leads to the following definition.

\begin{definition}[First pre-canonical basis]\label{First}
For any $\lambda \in X^+$ we  define
\begin{equation*} \label{defin elements N1}
 	\bfN_{\lambda}^{1} :=\frac{1}{\pi_{W^\lambda}(v^2)}\sum_{I\subset \Phi^+ } (-v^2)^{|I|}\widetilde{\undH}_{\lambda - \sum_{\alpha\in I}\alpha}.
 \end{equation*}
 where $\pi_{W^\lambda}(q)$ is the Poincar\'e polynomial of $W^\lambda:=\mathrm{stab}_{W_f}(\lambda)$.
\end{definition}

Notice that $|\Phi^+_{-w}|=\ell(w)$.
By \Cref{lambda-I}(\ref{part2}) we see that the coefficient of $\undH_\lambda$ in $\bfN_\lambda^1$ is
\[\frac{1}{\pi_{W^\lambda}(v^2)}\sum_{w\in W^\lambda}(-v^2)^{|\Phi^+_{-w}|}(-1)^{\ell(w)}=1.\]
Hence, by unitriangularity, the set $\{\bfN_\lambda^1\}$ is also a basis of $\Htil$. 

\begin{remark}\label{satake2}
It is easy to see that our definition of $\bfN^1_\lambda$  coincides, after applying the Satake isomorphism, with the definition of the Hall-Littlewood polynomials (compare e.g. with  \cite[Eq. (1)]{Ste3}). This immediately implies that $\bfN^1_\lambda$ is the standard basis element $\bfH_\lambda$, proving the first half of \Cref{12}. For the reader's convenience we reprove this fact in  \Cref{std=n1} avoiding the recourse to the Satake isomorphism.
\end{remark}

\begin{definition}
We call $\{\bfN_{\lambda}^{i} : \lambda\in  X^+\}$ the $i^{\text{th}}$ \emph{pre-canonical basis}.    
\end{definition}

\section{The first and second pre-canonical bases}\label{antiatomicsec}

From the definition of the $\bfN^i_\lambda$ it is clear that if $i> m$, where $m$ is the height of the highest root in $\Phi^+$, we have $\bfN^i_\lambda=\undH_\lambda$. The goal of this section is to show that also $\bfN^1_\lambda$ and $\bfN^2_\lambda$ are in fact familiar (and previously introduced) objects. Namely, $\bfN^1_\lambda$ and $\bfN^2_\lambda$ are respectively the standard basis and the $\bfN$-basis \cref{Nbasis} of $\Htil$.
We start by considering $\bfN^2$.

\begin{theorem}(Anti-atomic formula) \label{antiatomic}
For every $\lambda\in X^+$ we have
	\begin{equation}\label{antiatomiceq}
	\mathbf{N}_\lambda = \bfN_\lambda^2=\sum_{I\subset \Pgeq{2}} (-v^2)^{|I|}\widetilde{\undH}_{\lambda - \Sigma_I}.
	\end{equation}
\end{theorem}

Recall Definition \ref{phiw}.
Define $\Delta_w := \Delta \cap \Phi^+_w$ and $\Delta_{-w}:= \Delta \cap \Phi^+_{-w}$.
Following \cite[\S 24.1]{Hum2}, we define $\frX$ be the space of  $\bbQ[q]$-valued functions  $f$ on $X$ whose support
(defined to be the set of $x\in X$ for which $f(x)\neq 0$) is included in a finite
union of sets of the form $\{\lambda - \sum_{\alpha\in \Phi^+}k_\alpha \alpha,\; k_\alpha \in \bbZ_{\geq 0} \}$. (Such a set is the
set of weights occurring in a Verma module $Z(\lambda)$ ). 
In other words, an element of $\frX$ can be written as $\sum_{\mu\in X}c_\mu(q)e^\mu$, where $c_\mu(q)\in \bbQ[q]$ and such that the set of $\mu$ for which $c_\mu\neq 0$ is contained in a
finite
union of sets of the form $\{\lambda - \sum_{\alpha\in \Phi^+}k_\alpha \alpha, k_\alpha \in \bbZ_{\geq 0} \}$.
Then $\frX$ is a commutative $\bbQ[q]$ algebra. If $\alpha\in \Phi^+$, the element $(1-qe^{-\alpha})$ is invertible in $\frX$ and we have
\begin{equation}\label{-ainverse}
(1-qe^{-\alpha})^{-1}=\sum_{k\geq 0} q^k e^{-k\alpha}.
\end{equation}

Similarly, $(q-e^\alpha)\in \frX$ is also invertible and we have

\begin{equation}\label{ainverse}
(q-e^{\alpha})^{-1}= -e^{-\alpha}\sum_{k\geq 0} q^k e^{-k\alpha}.
\end{equation}

For any $\lambda\in X^+$, we consider the following element in $\frX$. 

\[\Theta(\lambda)=\sum_{w \in W_f} \frac{e^{w(\lambda)}\prod_{\alpha \in \Phi^{\geq 2}}(1-qe^{-w(\alpha)})}{\prod_{\alpha \in \Phi^+_w}(1-qe^{-w(\alpha)})\prod_{\alpha \in \Phi^+_{-w}}(q-e^{-w(\alpha)})}\]

For $f\in \frX$, we write $f|_{X^+}$ for its restriction to $X^+$ (i.e. if $f=\sum_{\mu \in X} c_\mu e^{\mu}$, then $f|_{X^+}=\sum_{\mu \in X^+} c_\mu e^{\mu}$).
We will expand $\Theta(\lambda)|_{X^+}$ in two different ways, and this will lead to our theorem.
The first expansion is given by \cref{exp1}.
\[
\Theta(\lambda)|_{X^+}=\displaystyle \sum_{\substack{I\subset \Phi^{\geq 2} \\ \lambda -\Sigma_I \ \mathrm{regular}}}(-q)^{|I|}(-1)^{\ell(w_{\lambda-\Sigma_I})} \sum_{\mu \in X^+}h_{\mu,\overline{\lambda-\Sigma_I}}(q^{\frac12})e^\mu.
\]
The second expansion is given by \cref{exp2}.
\[\Theta(\lambda)|_{X^+}=\displaystyle \sum_{\substack{\mu \in X^+ \\ \mu \leq \lambda \ } } q^{\hgt(\lambda-\mu)}e^\mu.
\]
For $\mu \in X^+$ such that $\mu \leq \lambda$, by comparing the coefficient of $e^\mu$ in \cref{exp1} and \cref{exp2} we obtain:
\begin{equation}\label{mucoeff}
    q^{\hgt(\lambda-\mu)}=\sum_{\substack{I\subset \Phi^{\geq 2} \\ \lambda -\Sigma_I \ \mathrm{regular}}}(-q)^{|I|}(-1)^{\ell(w_{\lambda-\Sigma_I})} h_{\mu,\overline{\lambda-\Sigma_I}}(q^{\frac12}).
\end{equation}

\begin{remark}
The rational function $\Theta(\lambda)$ is a $q$-deformation of the \emph{layer sum polynomials} 
\[ \Theta(\lambda)_{q=1} = \sum_{w \in W_f}\frac{e^{w(\lambda)}}{\prod_{\alpha\in \Delta}(1-e^{-w(\alpha)})}.\]
We also have  $\Theta(\lambda)_{q=1}=\sum e^{\mu}$, where the sum runs over all weights $\mu\in X$ lying in the convex hull of the orbit $W_f \cdot \lambda$. This was first obtained by Postnikov \cite[Theorem 4.3]{Pos} using Brion's formula for counting lattice points in rational polytopes \cite{Brion3} and later reproved by Sch\"utzer \cite{Sch} using root system combinatorics. 

Our strategy for proving \Cref{antiatomic} is based on Sch\"utzer's approach. 
We have carefully chosen the $q$-deformation in $\Theta(\lambda)$ so that, when restricted to dominant weights, it gives the desired $q$-deformation of the RHS of \cite[Eq. (10)]{Sch} (cf. \Cref{exp1}).
\end{remark}

\begin{proof}[Proof of \Cref{antiatomic}]
On the right side of \cref{antiatomiceq} we have
\begin{align*}
    \sum_{I\subset \Phi^{\geq 2} } (-v^2)^{|I|}\widetilde{\undH}_{\lambda - \Sigma_{I}}= & \sum_{\substack{I\subset \Phi^{\geq 2} \\ \lambda -\Sigma_I \ \mathrm{regular}}}(-v^2)^{|I|}(-1)^{\ell(w_{\lambda-\Sigma_I})}\undH_{\overline{\lambda - \Sigma_I}} \\
     = & \sum_{\substack{I\subset \Phi^{\geq 2} \\ \lambda -\Sigma_I \ \mathrm{regular}}}(-v^2)^{|I|}(-1)^{\ell(w_{\lambda-\Sigma_I})}\sum_{ \substack{  \mu\leq \overline{\lambda-\Sigma_I} \\ \mu\in X^+  } } h_{\mu,\overline{\lambda-\Sigma_I}}(v)\bfH_\mu.  
\end{align*}

Hence, for every $\mu\in X^+$ such that $\mu\leq \lambda$, the coefficient of $\bfH_\mu$ in the right side is 
\begin{equation}\label{mucoeff2}
\sum_{\substack{I\subset \Phi^{\geq 2} \\ \lambda -\Sigma_I \ \mathrm{regular}}}(-v^2)^{|I|}(-1)^{\ell(w_{\lambda-\Sigma_I})} h_{\mu,\overline{\lambda-\Sigma_I}}(v).
\end{equation}

Applying \cref{mucoeff} for $q=v^2$ we see that \cref{mucoeff2} is the same as $v^{2\hgt(\lambda-\mu)}$, the coefficient of $\bfH_\mu$  in $\bfN_\lambda$.
The identity  in \Cref{antiatomic} follows.
\end{proof}




\subsection{The First Expansion \texorpdfstring{of $\Theta(\lambda)$.}{}}\label{SecExp1}

We have
\begin{equation}\label{numerator}
\prod_{\alpha \in \Phi^{\geq 2}}(1-qe^{-w(\alpha)})=\sum_{I\subset \Phi^{\geq 2}}(-q)^{|I|}e^{-w(\Sigma_I)}.
\end{equation}
Using \cref{numerator},  \cref{-ainverse} and \cref{ainverse} we can rewrite $\Theta(\lambda)$ as 
\begin{equation}\label{theta1}
    \Theta(\lambda) = \sum_{I\subset \Phi^{\geq 2} } (-q)^{|I|}\sum_{w\in W_f} e^{w(\lambda-\Sigma_I)} \prod_{\alpha\in \Phi^+_w}\left(\sum_{k\geq 0} q^k e^{-kw(\alpha)}\right)\prod_{\alpha\in \Phi^+_{-w}}\left(-e^{w(\alpha)}\sum_{k\geq 0} q^k e^{kw(\alpha)}\right)
\end{equation}

Recall that for every $w\in W_f$ we have $|\Phi^+_{-w}|=\ell(w)$ and $w(\rho)-\rho = \sum_{\alpha \in \Phi^+_{-w}} w(\alpha)$. Hence, we have
\begin{equation}\label{wrho}
\prod_{\alpha \in \Phi^+_{-w}} -e^{w(\alpha)}=(-1)^{\ell(w)}e^{w(\rho)-\rho}.
\end{equation}

 

We can rewrite $\Theta(\lambda)$ once again as follows
\begin{equation*}
 \begin{array}{rl}
  \Theta(\lambda) =    &\displaystyle \sum_{I\subset \Phi^{\geq 2} } (-q)^{|I|}\sum_{w\in W_f} (-1)^{\ell(w)} e^{w(\lambda-\Sigma_I+\rho) -\rho} \prod_{\alpha\in \Phi^+_w}\left(\sum_{k\geq 0} q^k e^{-kw(\alpha)}\right)\prod_{\alpha\in \Phi^+_{-w}}\left(\sum_{k\geq 0} q^k e^{kw(\alpha)}\right)   \\
  & \\
                    =  & \displaystyle\sum_{I\subset \Phi^{\geq 2} } (-q)^{|I|}\sum_{w\in W_f} (-1)^{\ell(w)} e^{w(\lambda-\Sigma_I+\rho) -\rho} \prod_{\beta\in \Phi^+}\left(\sum_{k\geq 0} q^k e^{-k\beta}\right) \\
                    & \\
                    =   &\displaystyle \sum_{I\subset \Phi^{\geq 2} } (-q)^{|I|}\sum_{w\in W_f} (-1)^{\ell(w)} e^{w(\lambda-\Sigma_I+\rho) -\rho} \sum_{\beta\in Q^+} \kpf_q(\beta)e^{-\beta}.\qedhere
\end{array}   
\end{equation*}

 
 The first equality follows from \cref{wrho} applied to the last term of the right hand side of \cref{theta1}. For the second equality  replace $\beta= w(\alpha)$ and use the decomposition $\Phi^+=w(\Phi^+_w)\ \dot\cup\  (-w(\Phi^+_{-w})). $ The last equation follows directly from equation \cref{kpf}.
 The next Lemma implies that in previous sum only the subsets $I$ such that $\lambda-\Sigma_I$ is regular need to be considered.

 	


 \begin{lemma}
 	Assume $\nu\in X$ is singular, then 
 	\[\sum_{w\in W_f}(-1)^{\ell(w)}e^{w(\nu+\rho)} =0\]
 \end{lemma}
 \begin{proof}
 	Let $W^{\nu+\rho}$ be the stabilizer of $\nu+\rho$ in $W_f$. We denote by $W_{\nu+\rho}$ the set of minimal length representatives in $W_f/W^{\nu+\rho}$. Multiplication induces a length preserving bijection $W_{\nu+\rho}\times W^{\nu+\rho} \cong W_f$. We have
 	\[\sum_{w\in W_f}(-1)^{\ell(w)}e^{w(\nu+\rho)}= \sum_{x\in W_{\nu+\rho}}  (-1)^{\ell(x)}e^{x(\nu+\rho)} \sum_{y\in W^{\nu+\rho}}(-1)^{\ell(y)} \]

 It is enough to show that $\sum_{y\in W^{\nu+\rho}}(-1)^{\ell(y)}=0$. Notice that 
 \[\sum_{y\in W^{\nu+\rho}}(-1)^{\ell(y)}=|\{y \in W^{\nu+\rho}\mid \ell(y)\text{ even}\}|-|\{y \in W^{\nu+\rho}\mid \ell(y)\text{ odd}\}|.\]
 The group $W^{\nu+\rho}$ is a reflection subgroup (it is generated by the reflections in $W_f$ fixing $\nu+\rho$). If $\nu$ is singular then $W^{\nu+\rho}$ is non-trivial and contains a reflection $s$. Multiplication by $s$ induces a bijection between elements of even length and elements of odd length in $W^{\nu+\rho}$.
 \end{proof}

If  $\nu$ is now an arbitrary regular 
weight,  we have \[\sum_{w\in W_f} (-1)^{\ell(w)}e^{w(\nu+\rho)-\rho}=(-1)^{\ell(w_\nu)}\sum_{w\in W_f} (-1)^{\ell(w)}e^{w(\overline{\nu}+\rho)-\rho}\]

So, if $\lambda-\Sigma_I$ is regular, we have
\begin{align}\label{kpf2}
 &\sum_{w\in W_f} (-1)^{\ell(w)} e^{w(\lambda-\Sigma_I+\rho) -\rho} \sum_{\beta\in Q^+}\kpf_q(\beta)e^{-\beta}=\nonumber\\
= \,&(-1)^{\ell(w_{\lambda-\Sigma_I})}\sum_{w\in W_f} (-1)^{\ell(w)} e^{w(\overline{\lambda-\Sigma_I}+\rho) -\rho} \sum_{\beta\in Q^+}\kpf_q(\beta)e^{-\beta}\nonumber\\
= \,&(-1)^{\ell(w_{\lambda-\Sigma_I})}\sum_{\substack{\mu \in X \\ \mu\leq \overline{\lambda-\Sigma_I}}}\left(\sum_{w\in W_f} (-1)^{\ell(w)}  \kpf_q(w(\overline{\lambda-\Sigma_I}+\rho)-\rho-\mu)
\right)e^\mu.
\end{align}

The last equation is obtained by recalling that $\kpf_q(\alpha)$ is defined to be zero if $\alpha\in X\setminus Q^+$ and by noticing that for any $w\in W_f$ we have \[w\cdot \overline{\lambda-\Sigma_I} \leq \overline{\lambda-\Sigma_I}.\]
Finally, we restrict to $\Theta(\lambda)|_{X^+}$. Applying the definition of the Kostka-Foulkes polynomials \cref{KF} we obtain 
\begin{align}\label{exp1}
\Theta(\lambda)|_{X^+} & = \sum_{\substack{I\subset \Phi^{\geq 2} \\ \lambda -\Sigma_I \ \mathrm{regular}}}(-q)^{|I|}(-1)^{\ell(w_{\lambda-\Sigma_I})} \sum_{\mu \in X^+}K_{\overline{\lambda-\Sigma_I},\mu}(q)e^\mu\nonumber\\
& =\sum_{\substack{I\subset \Phi^{\geq 2} \\ \lambda -\Sigma_I \ \mathrm{regular}}}(-q)^{|I|}(-1)^{\ell(w_{\lambda-\Sigma_I})} \sum_{\mu \in X^+}h_{ \mu,\overline{\lambda-\Sigma_I}}(q^{\frac12})e^\mu.
\end{align}

\subsection{The Second Expansion \texorpdfstring{of $\Theta(\lambda)$}{}.}
Let $\Phi^{\geq 2}_{-w}:=\Phi^+_{-w}\setminus \Delta_{-w}$.
We start by rewriting $\Theta(\lambda)$ as
\begin{equation*}
    \begin{array}{rl}
     \Theta(\lambda)  =     & \displaystyle  \sum_{w \in W_f} \frac{e^{w(\lambda)}\prod_{\alpha \in \Phi^+_w\setminus \Delta_w}(1-qe^{-w(\alpha)})\prod_{\alpha \in \Phi^+_{-w}\setminus \Delta_{-w}}(1-qe^{-w(\alpha)})}{\prod_{\alpha \in \Phi^+_w}(1-qe^{-w(\alpha)})\prod_{\alpha \in \Phi^+_{-w}}(q-e^{-w(\alpha)})}  \\
     & \\
        = & \displaystyle  \sum_{w \in W_f} \frac{e^{w(\lambda)}}{\prod_{\alpha \in \Delta_w}(1-qe^{-w(\alpha)})\prod_{\alpha \in \Delta_{-w}}(q-e^{-w(\alpha)})}
\prod_{\alpha \in \Phi^{\geq 2}_{-w}}\frac{1-qe^{-w(\alpha)}}{q-e^{-w(\alpha)}}
    \end{array}
\end{equation*}
For every $\alpha \in \Phi^+_{-w}$ we have \[\frac{1-qe^{-w(\alpha)}}{q-e^{-w(\alpha)}}= \sum_{k\geq 0} p_k(q)e^{kw(\alpha)}\]
where $p_0(q)=q$ and $p_k(q)=q^{k+1}-q^{k-1}$ if $k\geq 1$. Now we can use \cref{-ainverse} and \cref{ainverse} to rewrite:

\begin{equation}\label{theta2}
\Theta(\lambda)=\sum_{w\in W_f}e^{w(\lambda)}\prod_{\alpha\in \Delta_w}\left(\sum_{k\geq 0} q^k e^{-kw(\alpha)}\right)\prod_{\alpha\in \Delta_{-w}}\left(-e^{w(\alpha)}\sum_{k\geq 0} q^k e^{kw(\alpha)}\right)\prod_{\alpha \in \Phi^{\geq 2}_{-w}}\left(\sum_{k\geq 0} p_k(q)e^{kw(\alpha)}\right)
\end{equation}

The terms $e^\mu$ which occur in the sum for a fixed $w\in W_f$ are for $\mu$ of the following form
\[\mu =w(\lambda)- \sum_{\alpha \in \Delta_w}k_\alpha w(\alpha) + \sum_{\beta \in \Delta_{-w}}(k_\beta+1) w(\beta) + \sum_{\gamma \in \Phi^{\geq 2}_{-w}}k_\gamma w(\gamma).\]
with $k_\alpha,k_\beta,k_\gamma \geq 0$.

Let $P^\lambda$ be the set of weights of an irreducible representation of highest weight $\lambda$ of the reductive group $G$ associated to the root datum.
In other words, $P^\lambda$ is the set of weights $\mu\in X$ such that $w(\mu) \leq \lambda$ for every $w\in W_f$.

The following Lemma is an adaptation of  \cite[Lemma 4.2]{Sch}.
\begin{lemma}\label{Xlam}
Let $\lambda\in X^+$. Assume that there is an element $\mu \in P^\lambda$ such that
	\begin{equation}\label{plam2}
	\mu =w(\lambda)- \sum_{\alpha \in \Delta_w}k_\alpha w(\alpha) + \sum_{\beta \in \Delta_{-w}}(k_\beta+1) w(\beta) + \sum_{\gamma \in \Phi^{\geq 2}_{-w}}k_\gamma w(\gamma)\end{equation}
	for some $k_\alpha,k_\beta,k_\gamma \geq 0$ and $w\in W_f$. Then $w=id$.
\end{lemma}
\begin{proof}
Since $\mu \in P^\lambda$,	we have $\lambda \geq w^{-1}(\mu)$, so
	$\lambda- w^{-1}(\mu) = \sum_{\alpha\in \Delta} n_\alpha \alpha$, for some $n_\alpha \in \bbZ_{\geq 0}$. 
	Applying $w$ we obtain
	\begin{equation}\label{plam1} w(\lambda) -\mu = \sum_{\alpha \in \Delta} n_\alpha w(\alpha).
	\end{equation}


	Now assume by contradiction that $w\neq id$.  Then $\Delta_{-w}\neq \emptyset$ and we choose $\beta \in \Delta_{-w}$. When we write $w(\lambda)-\mu$ in the basis $\{w(\alpha)\}_{\alpha \in \Delta}$ of $X_\bbR$, by \cref{plam1} the coefficient of $w(\beta)$ is $n_\beta \geq 0$.

	On the other hand, using \cref{plam2} we see that the coefficient of $w(\beta)$ in $w(\lambda)-\mu$ is  
	\[-k_\beta-1 -\sum_{\gamma \in \Phi^{\geq 2}_{-w}}k_\gamma [\gamma]_\beta<0,\] where $[\gamma]_\beta$ denotes the coefficient of $\beta$ when $\gamma$  is written in the basis of simple roots (clearly $[\gamma]_\beta\geq 0$). We get a contradiction, hence $w=id$.
\end{proof}

We restrict ourselves to consider $\Theta(\lambda)|_{X^+}$. Thanks to 
\cref{kpf2} we know that all the weights $e^\mu$ occurring in $\Theta(\lambda)|_{X^+}$ satisfy $\mu\leq \overline{\lambda-\Sigma_I}$ for some $I\subset \Pgeq{2}$ and therefore, by \Cref{lambda-I}, they satisfy $\mu\leq \lambda$. Since $\mu\in X^+$ we also have $\mu\in P^\lambda$.

Finally, from \Cref{Xlam}
it follows that 
only the term for $w=id$ contribute in \cref{theta2} to
the coefficient of $e^\mu$ in $\Theta(\lambda)$, for $\mu \in P^\lambda$. Notice that for $w = id$ we have $\Delta_{-w}=\Pgeq{2}_{-w}=\emptyset$. It follows that 

\begin{equation}\label{exp2}
\Theta(\lambda)|_{X^+}=\left. \left(e^\lambda \prod_{\alpha\in \Delta}\left(\sum_{k\geq 0} q^k e^{-k\alpha}\right)\right) \right|_{X^+}=\sum_{\substack{\mu\in X^+ \\ \mu\leq \lambda\  }}q^{\hgt(\lambda-\mu)}e^\mu.
\end{equation}

\subsection{The First Pre-Canonical basis.}

We can employ similar techniques to those of \Cref{antiatomicsec} to show that the first pre-canonical basis coincides with the standard basis, as mentioned in the first part of \Cref{12}. As pointed out in \Cref{satake2}, this can be deduced via the Satake transform, using the definition of the Hall-Littlewood polynomials. Here we give a proof which does not pass through Satake.
\begin{theorem}\label{std=n1}
    For every $\lambda \in X^+$ we have $\bfN^1_\lambda =\bfH_\lambda$.
\end{theorem}

We start by considering the following element of $\frX$. 
\[\Theta_1(\lambda)=\sum_{w \in W_f} \frac{e^{w(\lambda)}\prod_{\alpha \in \Phi^+}(1-qe^{-w(\alpha)})}{\prod_{\alpha \in \Phi^+_w}(1-qe^{-w(\alpha)})\prod_{\alpha \in \Phi^+_{-w}}(q-e^{-w(\alpha)})}.\]

We have
\[\prod_{\alpha \in \Phi^+}(1-qe^{-w(\alpha)})=\sum_{I\subset \Phi^+}(-q)^{|I|}e^{-w(\Sigma_I)}\]
The functions  $\Theta_1(\lambda)$ and $\Theta(\lambda)$ share the same denominator, so working as in  \Cref{SecExp1}
we obtain

\begin{equation}\label{Xi1}
    \Theta_1(\lambda)|_{X^+}=\sum_{\substack{I\subset \Phi^{+} \\ \lambda -\Sigma_I \ \mathrm{regular}}}(-q)^{|I|}(-1)^{\ell(w_{\lambda-\Sigma_I})} \sum_{\mu \in X^+}h_{\mu,\overline{\lambda-\Sigma_I}}(q^{\frac12})e^\mu.
\end{equation}

On the other hand, we have
\[ \Theta_1(\lambda) = \sum_{w \in W_f} e^{w(\lambda)}\prod_{\alpha \in \Phi^+_{-w}}\frac{1-qe^{-w(\alpha)}}{q-e^{-w(\alpha)}}=\sum_{w \in W_f} e^{w(\lambda)}\prod_{\alpha \in \Phi^+_{-w}}\left(\sum_{k\geq 0} p_k(q)e^{kw(\alpha)}\right).\]

For a fixed $w\in W_f$, the terms $e^\mu$ which occur in the sum  are for $\mu$ of the form
\[\mu = w(\lambda)+ \sum_{\alpha \in \Phi^+_{-w}}k_\alpha w(\alpha)  \ \ \ \mathrm{with}\ k_\alpha\in \bbZ_{>0}.\]
We need a slight modification of \Cref{Xlam}.

\begin{lemma}\label{Xlam2}
Let $\lambda\in X^+$. Assume there exists $\mu \in P^\lambda \cap X^+$ such that
	\begin{equation*}
	\mu =w(\lambda)+\sum_{\alpha \in \Phi^+_{-w}}k_\alpha w(\alpha)\end{equation*}
	for some $k_\alpha \geq 0$ and $w\in W_f$. Then $\mu=w(\lambda)=\lambda$ and $k_\alpha=0$ for all $\alpha\in \Phi^+_{-w}$.
\end{lemma}
\begin{proof}
Let $\mu$ be as in the statement. If $w=id$ then $\Phi^+_{-w}=\emptyset$, and the lemma follows. Assume now that $w\neq id$ and that $k_\gamma>0$ for some $\gamma\in \Phi^+_{-w}$. Then there exists $\beta\in \Delta_{-w}$ such that the coefficient $[\gamma]_\beta$ is strictly positive (otherwise  $\gamma \in \sum_{\alpha\in {\Delta_w}} \bbZ_{\geq 0}\alpha$  and $w(\gamma)\in \Phi^+$). Now, as in the proof of \Cref{Xlam}, by looking  at the coefficient of $w(\beta)$ in $w(\lambda)-\mu$ (again, in the basis $\{w(\alpha)\}_{\alpha \in \Delta}$ of $X_\bbR$), we deduce that it must be non-negative and negative, so we obtain a contradiction. It follows that  $k_\alpha=0$ for all $\alpha\in \Phi^+_{-w}$. Hence, we have $\mu = w(\lambda)\in X^+$. But $\lambda$ is the only dominant weight in its $W_f$-orbit, so $w(\lambda)=\lambda$. 
\end{proof}

Let $W^\lambda$ be the stabilizer of $\lambda$ in $W_f$.
We want to compute the coefficient of $e^\mu$ in $\Theta_1(\lambda)$ for $\mu\in X^+$. 
By \cref{Xi1} and \Cref{lambda-I} we know  that it is non-zero only if $\mu \leq \lambda$. Notice that $P^\lambda\cap X^+= \{\mu \in X^+ \mid \mu \leq \lambda\}$, so, by \Cref{Xlam2} we deduce that the coefficient of $e^\mu$ in $\Theta_1(\lambda)$ is non-zero only if $\mu=\lambda$.
Recall that $p_0(q)=q$. We conclude that 
\begin{equation}\label{Xi2} \Theta_1(\lambda)|_{X^+} = \sum_{w\in W^\lambda}q^{\ell(w)} e^\lambda =\pi_{W^\lambda}(q) e^\lambda.
\end{equation}

\begin{proof}[Proof of \Cref{std=n1}]
Comparing the coefficient of $e^\mu$ in \cref{Xi1} and \cref{Xi2} we see that for any $\mu \in X^+$ we have
\begin{equation}\label{mumu}
\delta_{\mu,\lambda}\pi_{W^\lambda}(q)=\sum_{I\subset \Phi^{+}}(-q)^{|I|}(-1)^{\ell(w_{\lambda-\Sigma_I})} \sum_{\mu \in X^+}h_{\mu,\overline{\lambda-\Sigma_I}}(q^{\frac12}).
\end{equation}
Then we conclude in a similar vein as in the proof of \Cref{antiatomic} (just after \cref{mucoeff}), i.e. the left hand side of \cref{mumu} (divided by $\pi_{W^\lambda}(q)$) gives the coefficient of $\bfH_\mu$ in $\bfH_\lambda$ and the right hand side (divided by $\pi_{W^\lambda}(q)$) gives the coefficient of $\bfH_\mu$ in $\bfN^1_\lambda$.
\end{proof}
The following corollary is the second part of \Cref{12} and it is an easy consequence of \Cref{antiatomic} and \Cref{std=n1}.

\begin{cor}\label{coro lowest decomposition}
For $\lambda \in X^+$ we have the equation
\[\bfN_{\lambda}^2   = \displaystyle \sum_{\substack{\mu \in X^+ \\ \mu \leq \lambda \ } }  q^{\hgt(\lambda-\mu)} \bfN_{\mu}^1.\]
\end{cor}

\begin{remark}
It seems natural to define for every $i\geq 1$ the rational function 
\[ \Theta_i(\lambda):=\frac{e^{w(\lambda)}\prod_{\alpha \in \Phi^{\geq i}}(1-qe^{-w(\alpha)})}{\prod_{\alpha \in \Phi^+_w}(1-qe^{-w(\alpha)})\prod_{\alpha \in \Phi^+_{-w}}(q-e^{-w(\alpha)})},\]
so that $\Theta(\lambda)=\Theta_2(\lambda)$
However, although the first expansion easily generalizes to every $i\geq 1$, and we have 
\[ \Theta_i(\lambda)|_{X^+} = \sum_{\substack{I\subset \Phi^{\geq i} \\ \lambda -\Sigma_I \ \mathrm{regular}}}(-q)^{|I|}(-1)^{\ell(w_{\lambda-\Sigma_I})} \sum_{\mu \in X^+}h_{ \mu,\overline{\lambda-\Sigma_I}}(q^{\frac12})e^\mu,\]
we do not have any interpretation for the second expansion when $i>2$.
\end{remark}


\section{Upper half decompositions in type A }\label{upp}

In this section we provide a closed formula for the decompositions of $\mathbf{N}_{\lambda}^{i+1}$  in terms of $\mathbf{N}_{\lambda}^{i}$ for all $n/2+1 \leq i \leq n $ in type $A_n$, thus proving \Cref{nhalf}. For the rest of this section we fix $n$ and  $i$.

\begin{definition}\label{M}
Let us define, for $A \subset \Phi$ and $\mu \in X$ the element
\[\mathbf{M}_\mu^A:= \sum_{I \subset A} (-q)^{|I|}\HTIL_{\mu-\Sigma_I}\in \mathcal{H}.\]
\end{definition}

\begin{lemma} \label{lemma preserve singularity and regularity}
 For $\lambda \in X $ and $s\in S_f$ a simple reflection of  $W_f$ we have the equation 
 \begin{equation*}  \label{eq action by s}
  \HTIL_{\lambda}=-\HTIL_{s\cdot \lambda}.	
 \end{equation*}
\end{lemma}
\begin{proof}
If $\lambda$ is singular, then $s\cdot \lambda$ is also singular. Then,  by definition, both $\HTIL_{\lambda}$ and $\HTIL_{s\cdot \lambda}$ vanish  and the lemma follows.\\
Assume now that $\lambda$ is regular. Then $s\cdot \lambda$ is also regular. We have $w_{s\cdot \lambda}=w_\lambda s$ since 
$ (w_{\lambda}s)\cdot (s\cdot \lambda) = w_\lambda\cdot \lambda\in X^+$.
Hence, $\ell(w_{s\cdot \lambda})=\ell(w_{\lambda})\pm 1$ and we conclude by using the definition of $\HTIL_{\lambda}$.
\end{proof}

We introduce two notations. 

\noindent
Let $\mu_1,\mu_2,\ldots , \mu_{n}$ be
the coordinates of $\mu \in X$ when expressed in the basis of fundamental weights, i.e. we have $\mu = \sum \mu_j \varpi_j$. It is not hard to prove  that $s_j\cdot \mu = \mu$ if and only if $\mu_j=-1$. 

\noindent
For $1\leq j \leq k\leq n$ define $\alpha_{j,k}:=\alpha_j+\alpha_{j+1}+\cdots +\alpha_{k}$ (all positive roots are of this form in type $A_n$). Notice that when written in the basis of fundamental weights, $\alpha_{j,k}$ has a $1$ in positions $j$ and $k$, a $-1$ in positions $j-1$ (if $1\leq j-1$) and $k+1$ (if $k+1\leq n$) and $0$ elsewhere.

\begin{prop}  \label{lemma X equal to zero} 
	Let $A \subset \Phi$  and $\mu \in X$. We have the equality  \[\mathbf{M}_{\mu }^A = - \mathbf{M}_{s_k \cdot \mu}^{s_k(A)},\] for all $1\leq k \leq n$.  In particular, if $A=s_k(A)$ and  $\mu_k=-1$, we have $\mathbf{M}_{\mu }^A =0$.
\end{prop}
\begin{proof}
The proposition is a direct consequence of \Cref{lemma preserve singularity and regularity} and the definition of the element $\bfM_{\mu}^A$.
\end{proof}

 For $1\leq j \leq n-i+1$  we set  $\gamma_j:=\alpha_{j,j+i-1}$. 
We define $\Gamma_{j}:= \Phi^{> i} \cup \{ \gamma_1,  \gamma_2 , \ldots , \gamma_j  \} $ and  $\Gamma_0 := \Phi^{> i}$.   To shorten notation we define
\begin{equation*}
	\mathbf{M}_{\lambda }^j : =  \mathbf{M}_{\lambda }^{\Gamma_{j}}.
\end{equation*}

\begin{lemma}  \label{lema descomposicion M facil}
	Let $\lambda  \in X^+$. For $j\geq 1$, if  $\lambda_{j+i-1}= 0$ then
\begin{equation} \label{eq lemma one}
\mathbf{M}_{\lambda }^j =\mathbf{M}_{\lambda }^{j-1 }.	
\end{equation}
\end{lemma}

\begin{proof}
	The disjoint union $\Gamma_j=\Gamma_{j-1}\cup \{ \gamma_j \}$ give us
\begin{equation}  \label{eq proof lemma one}
	\mathbf{M}_{\lambda }^j  =   \mathbf{M}_{\lambda }^{j-1}  - q  \bfM_{\lambda - \gamma_j}^{{j-1}}. 
\end{equation}	
On the other hand, we have $s_{j+i-1}(\Gamma_{j-1}) = \Gamma_{j-1}$ for all  $1\leq j \leq n-i +1$. As by  hypothesis $\lambda_{j+i-1}= 0$, we can use  \Cref{lemma X equal to zero}  to conclude that $\bfM_{\lambda - \gamma_j}^{{j-1}}=0$.
Therefore, \cref{eq proof lemma one} reduces to \cref{eq lemma one} and the Lemma is proved. 
\end{proof}

 \begin{lemma}  \label{lema descomposicion M dificil}
	Let $\lambda \in X^+ $. For $j\geq 1,$ if $\lambda_{j+i-1}>0$ then
	\begin{equation*} \label{decomposition cases}
		\mathbf{M}_{\lambda }^j  = 	\mathbf{M}_{\lambda }^{j-1} - q^{r-j+1} \mathbf{M}_{ \lambda -  \sum_{t=j}^{r} \gamma_t }^{j-1},	
	\end{equation*}
where $j\leq r \leq n-i+1$ is the smallest integer such that $ \lambda_{r} >0 $. If such an integer does not exist then $\mathbf{M}_{\lambda }^j  = 	\mathbf{M}_{\lambda }^{j-1} $. 
\end{lemma}

\begin{proof}
Let us first assume that the integer $r$ does exist. If $r=j$ then  $\lambda -\gamma_j\in X^+$ (recall that by hypothesis $\lambda_{j+i-1}>0$) and we conclude by \cref{eq proof lemma one}. Thus we assume that $r>j$.  
We define $B_j=\Gamma_{j-1}\setminus \{\alpha_{j,j+i}\}$ and for $j<k\leq n-i$ we define recursively  $B_{k}=B_{k-1} \setminus \{\alpha_{k,k+i}\}$. At first glance, this definition might seem odd, but $B_k$ is defined in such a way in order to satisfy the equation $s_k(B_k)=B_k$ for every $k$ satisfying $j\leq k <r$ (the proof of that equality is elementary but lengthy, and is left to the reader). If $\mu := \sum_{t=j}^{r} \gamma_t$, we have 
\begin{equation*}
\begin{array}{lll}
\mathbf{M}_{\lambda - \mu }^{j-1}   & = &\mathbf{M}_{\lambda - \mu }^{B_j}  -q  \mathbf{M}_{\lambda - \mu -\alpha_{j,j+i}  }^{B_j}      \\
\mathbf{M}_{\lambda -\mu }^{B_j}	  & = & \mathbf{M}_{\lambda - \mu }^{B_{j+1}} -q  \mathbf{M}_{\lambda - \mu -\alpha_{j+1,j+1+i} }^{B_{j+1}}     \\
\mathbf{M}_{\lambda - \mu }^{B_{j+1}}   & = & \mathbf{M}_{\lambda - \mu }^{B_{j+2}} -q  \mathbf{M}_{\lambda - \mu -  \alpha_{j+2,j+2+i} }^{B_{j+2}}    \\
\quad   \vdots                        & =  & \qquad  \vdots  \\
\mathbf{M}_{\lambda -  \mu }^{B_{r-2}}	   & = & \mathbf{M}_{\lambda - \mu }^{B_{r-1}} -q  \mathbf{M}_{\lambda - \mu -  \alpha_{r-1,r-1+i}  }^{B_{r-1}}   
\end{array}	
\end{equation*}
We sum all the equations above and after cancelling out similar terms we obtain 
\begin{equation*}
\mathbf{M}_{\lambda -\mu}^{j-1}  = 	\mathbf{M}_{\lambda - \mu }^{B_{r-1}} -q \sum_{k=j}^{r-1}  \bfM_{\lambda- \mu -\alpha_{k,k+i} }^{B_k} .
\end{equation*}
Let $k$ be an integer satisfying $j\leq k <r$.   By minimality of $r,$ the $k$-th coordinate of $\lambda -\mu -\alpha_{k,k+i}$ when written in terms of the  fundamental weights is equal to $-1$. Therefore, as $s_k(B_k)=B_k$ we use \Cref{lemma X equal to zero} to conclude that $\bfM_{\lambda-\mu - \alpha_{k,k+i}}^{B_k}=0$. It follows that 

\begin{equation}\label{rmo}
		\mathbf{M}_{\lambda - \mu }^{j-1} = \mathbf{M}_{\lambda - \mu }^{B_{r-1}} . 
\end{equation}
Since $\lambda_j=0$ and $s_j(B_{j})=B_j$, using \Cref{lemma X equal to zero} one can prove the second equality: 
\begin{equation} \label{from j-1 to Bj}
	\bfM_{\lambda- \gamma_j}^{j-1} =     \bfM_{\lambda- \gamma_j}^{B_j} -q \bfM_{\lambda - \gamma_j -\alpha_{j,j+i}}^{B_j} =\bfM_{\lambda- \gamma_j}^{B_j} +q \bfM_{\lambda - \gamma_j -\gamma_{j+1}}^{B_j} . 
\end{equation}
Similarly, since $\lambda_k=0$ and $s_k(B_k)=B_k$ for all $j<k<r$,  \Cref{lemma X equal to zero} implies 

 \begin{equation}  \label{from Bk-1 to Bk}
 	\bfM_{\lambda - \sum_{t=j}^k  \gamma_t }^{B_{k-1}  } =\bfM_{\lambda - \sum_{t=j}^k  \gamma_t}^{B_k} -q\bfM_{\lambda - \alpha_{k,k+i} - \sum_{t=j}^k  \gamma_t}^{B_k}  =  \bfM_{\lambda - \sum_{t=j}^k  \gamma_t}^{B_k} +q \bfM_{\lambda -  \sum_{t=j}^{k+1}  \gamma_t}^{B_k} 
 \end{equation}
Therefore,  \cref{from j-1 to Bj} and a repeated application of \cref{from Bk-1 to Bk} yields
\begin{equation*}
	\bfM_{\lambda- \gamma_j}^{j-1} = 	q^{r-j}  \bfM_{\lambda - \mu }^{B_{r-1}}   + \sum_{k=j}^{r-1} q^{k-j}  \bfM_{\lambda - \sum_{t=j}^k \gamma_t}^{B_k}.
\end{equation*}
Once again, we can use \Cref{lemma X equal to zero}  to conclude that $\bfM_{\lambda - \sum_{t=j}^k \gamma_t}^{B_k} =0$ for all $j\leq k <r$. It follows that 
\begin{equation*}
	\bfM_{\lambda- \gamma_j}^{j-1}  = q^{r-j}  \bfM_{\lambda - \mu}^{B_{r-1}}.
\end{equation*}
Finally, we obtain
\begin{equation*}
\mathbf{M}_{\lambda }^j = \mathbf{M}_{\lambda }^{j-1}  - q   \bfM_{\lambda - \gamma_j}^{j-1}  = \mathbf{M}_{\lambda }^{j-1}  -  q^{r-j+1}  \bfM_{\lambda - \mu}^{B_{r-1}} = \mathbf{M}_{\lambda }^{j-1}  -  q^{r-j+1}       \mathbf{M}_{\lambda - \mu}^{j-1},
\end{equation*}
where the last equality is \cref{rmo}. We have proved the lemma when the integer $r$ exists.\\
We now assume  $r$ does not exist. This means that $\lambda_k =0$ for all $j\leq k \leq n-i+1$. Arguing as before, we get
\begin{equation*}
	\bfM_{\lambda- \gamma_j}^{j-1} = 	q^{n-i+1-j}  \bfM_{\lambda -\nu}^{B_{n-i}},
\end{equation*}
where $\nu:= \sum_{t=j}^{n-i+1} \gamma_{t}$. It is again a lengthy but easy problem to prove that $s_{n-i+1}(B_{n-i}) = B_{n-i}$. Furthermore, the $(n-i+1)$-th component of $\lambda - \nu$ when written in terms of the basis of fundamental weights is $-1$. By applying again \Cref{lemma X equal to zero}  we conclude that $\bfM_{\lambda -\nu}^{B_{n-i}}=0$. Finally, we have
\begin{equation*}
\mathbf{M}_{\lambda }^j=\mathbf{M}_{\lambda }^{j-1}  - q 	\bfM_{\lambda- \gamma_j}^{j-1}
       =\mathbf{M}_{\lambda }^{j-1} - q^{n-i+2-j}  \bfM_{\lambda -\nu}^{B_{n-i}}
       =\mathbf{M}_{\lambda }^{j-1},
\end{equation*}
as required.
\end{proof}

We recall a definition given in the introduction. 
\begin{definition}
    Let $1\leq i \leq n$. Let $\lambda , \mu \in X^+$. We write $\lambda \geq_i \mu$ if $\lambda - \mu$ can be written as a positive integral linear combination of the elements of $\Phi^i$. 
\end{definition}

\begin{theorem} \label{teo explicit highest decompsitions}
    Let $n\geq 2$ and  $n/2+1 \leq  i \leq n $.  We have the equation 
    \begin{equation*}
        \bfN_{\lambda}^{i+1} =  \displaystyle \sum_{\mu \leq_i \lambda } q^{\frac{1}{i} \hgt (\lambda - \mu) } \bfN_{\mu}^i
    \end{equation*}
    for all $\lambda \in X^+$. 
\end{theorem}

\begin{proof}
Let $\bar n = n-i+1$ and $\lambda\in X^{+}$. The hypothesis ensures that $\bar n<i$. Let  $1\leq j \leq \bar n$. Suppose that there exists $ j \leq r \leq \bar n$ such that $\lambda - \sum_{t=j}^r \gamma_t \in X^+$. In this case we define $R_j(\lambda) := \lambda - \sum_{t=j}^r \gamma_{t} $ where $r$ is minimal with the above property. If such an $r$ does not exist then $R_{j}(\lambda)$ is not defined. We stress that $R_{j}(\lambda)$ is defined if and only if $\lambda_{j+i-1}>0$ and at least one of the integers $\lambda_{j}, \lambda_{j+1}, \ldots ,\lambda_{\bar n}$ is greater than zero. With these notations \Cref{lema descomposicion M dificil} can be restated as follows 
\begin{equation} \label{eq restating Lemma decomposition }
   \bfM_{\lambda}^{j-1}=\bfM_{\lambda}^j+q^{ \frac{1}{i} \hgt(\lambda - R_{j} (\lambda))}  \bfM_{R_j(\lambda)}^{j-1}.  
\end{equation}
If $k$ is the maximal integer such that $R_j^k(\lambda)$  is defined then \Cref{lema descomposicion M facil} and \cref{eq restating Lemma decomposition } (applied  $k$ times) imply 
\begin{equation*}
	\bfM_{\lambda}^{j-1}  = \sum_{s=0}^k q^{\frac{1}{i} \hgt(\lambda - R_{j}^s(\lambda) )  }  \bfM_{R_{j}^s(\lambda)}^{j}.
\end{equation*}
Since $\bfM_{\lambda}^0=\bfN_{\lambda}^{i+1}$ and $\bfM_{\lambda}^{\bar n} = \bfN_{\lambda}^{i}$ we have that $\bfN_{\mu}^i  $ occurs in the decomposition of $\bfN_{\lambda}^{i+1}$ with coefficient $c_{\lambda}(\mu) q^{\frac{1}{i} \hgt(\lambda -\mu)}$ where $c_{\lambda }(\mu)$ is the number of tuples $(a_{1},a_{2},\ldots , a_{\bar n})\in (\bbZ_{\geq 0})^{\bar n}$ such that 
\begin{equation*}
	R_{\bar n}^{a_{\bar n}} \cdots R_{2}^{a_{2}} R_{1}^{a_{1}} (\lambda) =\mu . 
\end{equation*}  
It is clear that if $\mu \not \leq_i \lambda$ then $c_{\lambda}(\mu)=0$. Thus in order to prove the Theorem it is enough to show that if $\mu \leq_i \lambda$ then such a sequence exists and it is unique.

\vspace{0.1cm}

\noindent
\textbf{\emph{Uniqueness.}} Assume there are two sequences $(a_1,a_2, \ldots ,a_{\bar n})$ and $(a_1',a_2',\ldots , a_{\bar n}')$ such that
\begin{equation}\label{uniqueness}
	R_{\bar n}^{a_{\bar{n}}} \cdots R_{2}^{a_{2}} R_{1}^{a_{1}} (\lambda)  = R_{\bar n}^{a_{\bar{n}}'} \cdots R_{2}^{a_{2}'} R_{1}^{a_{1}'} (\lambda) .
\end{equation}
 By explicitly calculating $\sum_{t=j}^r\gamma_t$ in the basis of fundamental weights,  we see that $(R_j(\lambda))_k$ (i.e. the $k^{\mathrm{th}}$ component in the decomposition of $R_j(\lambda)$ in the fundamental weights) can only differ from $\lambda_k$ if \[k\in \{j-1,j,j+1,\ldots, \bar n\}\cup \{j+i-1,j+i,j+i+1,\ldots,n\}.\] Remark that the union is disjoint because $\bar n<i$  and that  $(R_j(\lambda))_{j+i-1}=\lambda_{j+i-1}-1$. Using this we can compare  the coefficient of $\varpi_i$ on both sides of \cref{uniqueness} and obtain $\lambda_{i}-a_1 = \lambda_{i} -a_1'$, and therefore $a_1=a_1'$. Now  assume that for some $1\leq j <\bar n$ we have $a_s=a_s'$ for all $ 1\leq s \leq j$. Then, by comparing the coefficient of $\varpi_{i+j}$ on both sides of \cref{uniqueness} we obtain $c-a_{j+1}=c-a_{j+1}'$, where $c$ is the $(i+j)$-th coordinate of $R_{j}^{a_{j}} \cdots R_{2}^{a_{2}} R_{1}^{a_{1}} (\lambda)  = R_{j}^{a_{j}'} \cdots R_{2}^{a_{2}'} R_{1}^{a_{1}'} (\lambda) $. Thus $a_{j+1}=a_{j+1}'$. By induction we conclude that both sequences are the same. 

\vspace{0.2cm}

\noindent
\textbf{\emph{Existence.}}  We will prove existence by induction in $\hgt (\lambda - \mu ).$ In the base case $\hgt (\lambda - \mu )=0$ (i.e. $\lambda = \mu$) there is nothing to prove.

Let $\lambda - \mu = \sum_{j=1}^{\bar n} m_{j}\gamma_{j}$ and suppose that existence is proved for any $\lambda', \mu'$ such that $\hgt (\lambda' - \mu' )<\hgt (\lambda-\mu)$.  Let $\lambda \neq \mu $ and let $s$ be minimal such that $m_s>0$. We have 
\begin{equation} \label{eq existence lambda -mu }
    \lambda -\sum_{j=s}^{\bar n} m_{j}\gamma_{j} =\mu .
\end{equation}
 Since $\mu \in X^+$ the above equality implies that $\lambda_{s+i-1}>0$ and that at least one of the following integers $\lambda_{s}, \lambda_{s+1}, \ldots , \lambda_{\bar n}$  is greater than zero. In particular, $R_{s}(\lambda)$ is defined.\\
We claim that $\mu \leq_i R_s(\lambda)$. Indeed, if
\begin{equation*}
    R_{s}(\lambda) = \lambda - \sum_{j=s}^r\gamma_{j}   
\end{equation*}
for some $s\leq r \leq \bar n$, then  the minimality of $r$ implies $\lambda_{s}=\lambda_{s+1}= \cdots = \lambda_{r-1}=0$. Since $\mu$ is dominant, the above equalities and \cref{eq existence lambda -mu }  imply that $m_j>0$ for all $s \leq j \leq r$ (by observing the $s^{\mathrm{th}}$ component,  $m_s>0$ and $\lambda_s=0$ imply that $m_{s+1}>0$. Then observe the $(s+1)^{\mathrm{th}}$ component and so on). Hence
\begin{equation*}
	R_{s}(\lambda) - \mu  = (	R_{s}(\lambda) - \lambda ) + (\lambda - \mu  )  =\sum_{j=s}^r(m_j-1)\gamma_{j}  + \sum_{j=r+1}^{\bar n} m_j \gamma_{j}
\end{equation*}
and  the claim follows.\\
We notice that $\hgt(R_s(\lambda) -\mu) <\hgt (\lambda - \mu )$. Then, by induction hypothesis, there exists a sequence  $(b_1,b_2, \ldots ,b_{\bar n})\in \mathbb{Z}_{\geq 0}^{\bar n }$ such that
\begin{equation} \label{eq existence A}
	R_{\bar n}^{b_{\bar{n}}} \cdots R_{2}^{b_{2}} R_{1}^{b_{1}} (R_{s}(\lambda)) =\mu .
\end{equation} 
By \cref{eq existence lambda -mu } we have $\lambda_k=\mu_k$ for all $i\leq k <s+i-1$ (notice that substracting $\gamma_j$ for $s\leq j\leq \bar n$ does not affect the $k^{\mathrm{th}}$ component for  $i\leq k <s+i-1$).  Then \cref{eq existence A} implies $b_1=b_2=\cdots = b_{s-1}=0$. Therefore, the desired sequence is given by $a_s=b_s+1$ and $a_j =b_j$ for $j\neq s$.
\end{proof}


\section{Explicit decompositions in types \texorpdfstring{$\tilde{A}_3$}{A\_3} and \texorpdfstring{$\tilde{A}_4$}{A\_4}}\label{last}

\subsection{Type \texorpdfstring{$\tilde{A}_3$}{A\_3}.}

Throughout this section we fix $n=3$.  Let  $\lambda = a\varpi_1 + b\varpi_2 + c\varpi_3 \in X^+$.   \Cref{teo explicit highest decompsitions} gives the decomposition of  $ \underline{\bfH}_{\lambda } = \bfN_{\lambda}^4$ in terms of $\{\bfN_{\mu }^3  \mid \mu \in X^+   \}$ (see \cref{eq n4 en n4 intro}).
Furthermore, \Cref{coro lowest decomposition} provides the decomposition of $\bfN_{\lambda}^2$ in terms of $\{\bfN_{\mu }^1  \mid \mu \in X^+   \}$. Therefore, to complete the description of all the decompositions we need to explain  the decomposition of $\bfN_{\lambda}^3$ in terms of $\{ \bfN_{\mu}^2 \mid \mu \in X^+ \}$. We proceed indirectly by first finding the inverse decomposition.  It follows from Definition \ref{definition pre-canonical basis} that 
\begin{equation}\label{generic decomposition}
    \bfN_{\lambda}^2 = \bfN^3_{\lambda} -q \bfN^3_{\lambda - \alpha_{12}} -q\bfN^3_{\lambda -\alpha_{23}} +q^2\bfN^3_{\lambda -\alpha_{12}-\alpha_{23}}
\end{equation}
for $a\geq 1$, $b\geq 2$ and $c\geq 1$. We refer to the above as the generic decomposition and to any other case as non-generic.  The following lemma provides the non-generic decompositions for $\bfN^2_{\lambda}$ in terms of $\{\bfN^3_{\mu}  \mid \mu \in X^+  \} $. 

\begin{lemma}
 Let  $\lambda = a\varpi_1 + b\varpi_2 + c\varpi_3 \in X^+$. In type $\tilde{A}_3$ the non-generic decomposition for $\bfN^2_{\lambda}$ in terms of $\{\bfN^3_{\mu} \mid  \mu \in X^+   \}$ is given in \Cref{tab:decom N1 in N2}.
\begin{table}[!h]
    \centering
    \[
    \begin{array}{||c||c|c|c||l|}\hline  
\mbox{Row} &  a     &    b     &     c     &    \mbox{Decomposition}     \\ \hline   \hline
      1    &   0    & \geq 1   & \geq 1    &    \bfN^3_{\lambda}  -q\bfN^3_{\lambda-\alpha_{23}}                  \\     \hline 
      2    &  0     &  0       &  \geq 0   &    \bfN^3_{\lambda}                \\   \hline
      3    & 0      &  1       & 0         &    \bfN^3_{\lambda}                \\   \hline
      4    & 0      &  \geq 2  & 0         &    \bfN^3_{\lambda} - q^2  \bfN^3_{\lambda-\alpha_{12}-\alpha_{23}}                \\   \hline
      5    & \geq 1 & \geq 1   & 0         &    \bfN^3_{\lambda}  -q\bfN^3_{\lambda-\alpha_{12}}                  \\   \hline
      6    & \geq 0 &   0      & 0         &    \bfN^3_{\lambda}             \\   \hline
      7    & \geq 1 &   1      &  \geq 1   &    \bfN^3_{\lambda} -q\bfN^3_{\lambda -\alpha_{12}} -q\bfN^3_{\lambda - \alpha_{23}}            \\   \hline
      8    & \geq 1 &   0      &  \geq 1   &    \bfN^3_{\lambda} -q^2\bfN^3_{\lambda -\alpha_{13}}            \\   \hline
    \end{array}
    \]
    \caption{Decomposition of $\bfN^2_{\lambda} $ in terms of $\{\bfN_{\mu}^3 \mid \mu  \in X^+\}  $ for $\tilde{A}_3$.}
    \label{tab:decom N1 in N2}
\end{table}
\end{lemma}
\begin{proof}
The result follows by a case-by-case analysis. We only prove here the decomposition given by Row 1 in Table \ref{tab:decom N1 in N2}. So we are in the case $a=0$, $b\geq 1$ and $c \geq 1$. 
Let $A=\Pgeq{2}\setminus\{\alpha_{12}\}$. Notice that $s_1(A)=A$. By applying \Cref{lemma X equal to zero} we obtain $
\bfN_{\lambda}^2= \bfM_{\lambda}^A -q\bfM_{\lambda - \alpha_{12}}^A = \bfM_{\lambda}^A = \bfN^3_{\lambda}  -q\bfN^3_{\lambda-\alpha_{23}},
$
which coincides with the decomposition predicted by Row 1 in Table \ref{tab:decom N1 in N2}.
\end{proof}


Given $\lambda \in X^+$ we define $\Nhat_{\lambda}^3$ as the right hand side of \cref{eq intro N3 in N2}.

\begin{lemma}  \label{lemma decomposition hat}
 Let $\lambda \in X^+$. The decomposition of $\bfN_{\lambda}^2$ in terms of $\{ \Nhat_{\mu}^3  \mid  \mu \in X^+  \} $ is the same as the decomposition in terms of  $\{ \bfN_{\mu}^3  \mid \mu \in X^{+}\}$. 
\end{lemma}

\begin{proof}
We only need to check that the decomposition of $\bfN_{\lambda}^2$ in terms of $\{\Nhat_{\mu}^3 \mid \mu \in X^{+} \} $ coincides with the one given in \cref{generic decomposition} and \Cref{tab:decom N1 in N2}. This requires a case-by-case analysis. We leave the details to the reader since in \S\ref{a4} we treat in full detail the similar but harder situation in type $\tilde{A}_4$.
\end{proof}

\begin{theorem} \label{teo decom in type A3 }
For all $\lambda \in X^{+}$ we have $\bfN_{\lambda}^3 = \Nhat_{\lambda}^3$. Therefore, \cref{eq intro N3 in N2} provides the decomposition of $\bfN_{\lambda}^3$ in terms of $\{\bfN_{\mu}^2 \mid \mu \in X^{+} \}$. 
\end{theorem}

\begin{proof}
This is a direct consequence of  \Cref{lemma decomposition hat}. 
\end{proof}

\subsection{Type \texorpdfstring{$\tilde{A}_4$}{A\_4}.} \label{a4}
In this section we fix $n=4$. We stress that $\undH_{\lambda} = \bfN_{\lambda}^5$ for all $\lambda\in X^+$.  \Cref{teo explicit highest decompsitions} provides the decomposition of $\bfN_{\lambda}^5$  in terms of $\{\bfN_{\mu}^4 \mid \mu \in X^+\}$ and of $\bfN_{\lambda}^4$ in terms of $\{\bfN_{\mu}^3 \mid \mu \in X^+\}$. On the other hand, the decompsition of $\bfN_{\lambda}^2$ in terms of $\{\bfN_{\mu }^1  \mid \mu \in X^+   \}$ is covered by \Cref{coro lowest decomposition}. Thus, such as in the previous section, we only need to specify the decomposition of  $\bfN_{\lambda}^3$ in terms of $\{\bfN_{\mu}^2 \mid \mu \in X^{+}\}$. \\
Let $\lambda,\mu$ in $X^+$. We write $\mu \preccurlyeq_2 \lambda$ if $\lambda - \mu$ can be written as an integral non-negative  linear combination of the elements of $\Pgeq{2}$. Given $\mu \preccurlyeq_2  \lambda$   we denote by $\mcL_{\lambda}(\mu)$ the set of all non-negative linear combinations of the elements of $\Pgeq{2}$ equal to  $\lambda - \mu$. We denote a linear combination $L\in \mcL_{\lambda}(\mu)$ by $L=(l_{ij})$, where $l_{ij}\in \mathbb{N}$ is the coefficient of $\alpha_{ij}$ in $L$.
%
Given $L=(l_{ij}) \in  \mcL_{\lambda}(\mu) $ we define its \emph{degree} as
\begin{equation*}
	\deg (L): =  \sum_{1\leq i < j \leq 4}   l_{ij}( \hgt (\alpha_{ij} ) -1). 
\end{equation*}
On the other hand, we set $\nu_0(L) =\lambda$ and for $k=1,2,3$ we recursively define 
\begin{equation*}
	\nu_{k}(L)=  \nu_{k-1}(L) - \sum_{j-i=k} l_{ij}\alpha_{ij}. 
\end{equation*}
 We define integers $\nu_{k}^{j}(L)$ by the equation
$ \nu_{k}(L)= \nu_k^1(L) \varpi_1 + \nu_k^2(L)\varpi_2 + \nu_k^3(L) \varpi_3 + \nu_{k}^4(L)\varpi_4$.

\begin{definition}\label{defi admisible}
	We say that an element $L\in \mcL_{\lambda}(\mu)$ is \emph{admissible} if it satisfies the following conditions
	\begin{itemize}
		\item $v_{k}(L) \in X^+$ for all $1\leq k\leq 3$;
		\item If $l_{13}\neq 0$ then $\nu_1^2(L)=0$;
		\item If $l_{24}\neq 0 $ then $\nu_1^3(L)=0$;
		\item If $l_{14}\neq 0 $ then $\nu_2^2(L)=\nu_2^3(L)=0$;
	\end{itemize}
We denote the set of all admissible $L$ by 	 $\mcL_{\lambda}^a(\mu)$ and define
$r_{\lambda}(\mu)= \displaystyle \sum_{L\in \mcL_{\lambda}^a(\mu)} q^{\deg (L)}  .$
Finally, we define
\begin{equation} \label{defi N3 hat in type A4}
  \hat{\bfN}_{\lambda}^3=\displaystyle \sum_{\mu \preccurlyeq_2 \lambda   } r_{\lambda}(\mu) \bfN_{\mu}^2.  
\end{equation}
\end{definition}


\begin{theorem} \label{teo decomposition in A4}
    For all $\lambda \in X^+$ we have $  \bfN_{\lambda}^3 = \hat{\bfN}_{\lambda}^3  $. Consequently, \cref{defi N3 hat in type A4} provides the decomposition of $\bfN_{\lambda}^3$ in terms of $\{\bfN_{\mu}^2 \mid \mu \in X^{+} \}$. 
\end{theorem}

The proof of \Cref{teo decomposition in A4} will follow the same lines as the proof of \Cref{teo decom in type A3 }, this is, we will eventually show that the expansion of $\bfN_{\lambda}^2$ in terms of $\{\bfN_{\mu}^3 \mid \mu \in X^+\}$ is the same as the expansion in terms of $\{\hat{ \bfN_{\mu}^3} \mid \mu \in X^+\}$. As in type $\tilde{A}_3$ we have a generic case. 

\begin{prop}  \label{Propo generic case A4}
  Let $\lambda = a\varpi_1 + b\varpi_2 +c\varpi_3+ d\varpi_4 \in X^+$. Assume that $a\geq 1$, $b\geq 2$, $c\geq 2$ and $d\geq 1$. Then, we have
 \begin{align}
   \label{generic in type A4}
     \bfN_{\lambda }^2 & = \displaystyle \sum_{J\subset \Phi^2} (-q)^{|J|} \bfN_{\lambda-\Sigma_{J} }^3 \\
    \label{generic in type A4 hat }
     \bfN_{\lambda }^2 &= \displaystyle \sum_{J\subset \Phi^2} (-q)^{|J|} \hat{\bfN}_{\lambda-\Sigma_{J}}^3. 
 \end{align} 
\end{prop}
\begin{proof}
The conditions imposed on $a,b,c$ and $d$, are equivalent to asking that $\lambda - \Sigma_J\in X^+$ for all $J\subset \Phi^2$. Therefore, \cref{generic in type A4} follows easily from \Cref{definition pre-canonical basis}. \\
We now prove \cref{generic in type A4 hat }. We define elements $m_{\mu} \in \bbZ[v,v^{-1}]$ by the equation
\begin{equation*}
    \displaystyle \sum_{J\subset \Phi^2} (-q)^{|J|} \hat{\bfN}_{\lambda-\Sigma_{J}}^3 = \sum_{\mu \in X^+} m_\mu \bfN_{\mu}^2. 
\end{equation*}  
By \cref{defi N3 hat in type A4} it is clear that $m_\lambda =1$ and  that $m_\mu =0$ if $\mu \not \preccurlyeq_2 \lambda $. We fix $\mu \preccurlyeq_2 \lambda$ with $\mu \neq \lambda$. 

Assume that  $\mcL_{\lambda}^a(\mu)\neq \emptyset$. Let $J\subset \Phi^2$. Given  $L =(l_{ij})\in \mcL_{\lambda -\Sigma_J}(\mu)  $ we define $L'=(l'_{ij})\in \mcL_{\lambda }(\mu)$ as $l'_{ij} \coloneqq l_{ij}+ 1 $ if $\alpha_{ij}\in J$, and $l'_{ij} \coloneqq l_{ij}$, otherwise.
This defines a map $F:\mcL_{\lambda -\Sigma_J}(\mu) \mapsto \mcL_{\lambda }(\mu)$.  Moreover, as $J\subset \Phi^2$ one has that $\deg (L')=\deg (L)+|J|$. Notice that, since $J\subset \Phi^2$, we have that $\nu_k(L)=\nu_k(L')$ for $1\leq k\leq 3$ and that $l_{13}=l'_{13}$, $l_{24}=l'_{24}$ and $l_{14}=l'_{14}$. It follows that $F$ preserves admissible elements, i.e. it restricts to a map
$F:\mcL_{\lambda -\Sigma_J}^a(\mu) \mapsto \mcL_{\lambda }^a(\mu)$ that is clearly injective. The image $F(\mcL_{\lambda -\Sigma_J}^a(\mu))$  is the set
\begin{equation*}
    \mcL_{\lambda }^a(\mu)^J:=\{ L=(l_{ij})\in \mcL_{\lambda }^a(\mu)  \mid l_{ij}\neq 0 \mbox{ for all } \alpha_{ij}\in J  \}. 
\end{equation*}
From this it follows that \begin{equation}
\label{degL}
\sum_{L\in \mcL_{\lambda }^a(\mu)^{J}} q^{\deg(L)}=q^{|J|}r_{\lambda-\Sigma_J}(q).\end{equation}
On the other hand, by definition of admissibility we have
\begin{equation}\label{L}
    \mcL_{\lambda }^a(\mu) = \bigcup_{\alpha \in \Phi^2}  \mcL_{\lambda }^a(\mu)^{\{\alpha\}}.
\end{equation}
If in \cref{L} one takes the $q$-graded degree on both sides we obtain by \cref{degL} and the inclusion-exclusion principle: 


\begin{equation}\label{generic A4 AA}
0  =    \sum_{J\subset \Phi^2} (-q)^{|J|} r_{\lambda -\Sigma_J} (\mu). 
\end{equation}
(We use here that $\mcL_{\lambda }^a(\mu)^{J_1}\cap \mcL_{\lambda }^a(\mu)^{J_2} =  \mcL_{\lambda }^a(\mu)^{J_1\cup J_2}$ for any $J_1, J_2\subset \Phi^2$.)

\noindent
The right-hand side of \cref{generic A4 AA} is equal to $m_\mu$. Therefore, we have shown that $m_\mu =0$ if $\mcL_{\lambda}^a(\mu)\neq \emptyset$.

We now assume that $\mcL_{\lambda}^a(\mu)= \emptyset$. In this case we have $\mcL_{\lambda -\Sigma_J}^a(\mu)= \emptyset$ for all $J\subset \Phi^2$ (recall that $\mcL_{\lambda -\Sigma_J}^a(\mu)$ injects in $\mcL_{\lambda}^a(\mu)$).  We conclude that  $m_\mu =0$ in this case as well. \\
Summing up, we have proved that $m_\mu = \delta_{\lambda,\mu}$ (where $\delta_{\lambda,\mu}$ is a Kronecker delta). By the definition of $m_\mu$ this is equivalent to \cref{generic in type A4 hat }.
\end{proof}

Before we prove the non-generic decomposition of $\bfN_{\lambda}^2$, we record
in the following five lemmas some useful computations which are later needed several times. The proofs of all these lemmas are routine computations using \Cref{lemma X equal to zero} and the definition of the relevant $\bfN$ and $\bfM$ elements. For this reason we omit the proofs.

\begin{lemma}\label{lemma a=0}
Let $\lambda = a\varpi_1 + b\varpi_2 + c\varpi_3 +d\varpi_4 \in X^+$. Suppose that $a=0$. Then
\begin{equation}\label{eq a=0}
    \bfN_{\lambda}^2= \bfN_{\lambda}^3-q\bfM_{\lambda-\alpha_{23}}^{\Pgeq{3}}-q\bfM_{\lambda-\alpha_{34}}^{\Pgeq{3}}  +q^2 \bfM_{\lambda-\alpha_{23}-\alpha_{34}}^{\Pgeq{3}}.
\end{equation}
\end{lemma}

\begin{remark}
By definition, if $\lambda \in X^+$, we have $\bfN^k_\lambda=\bfM^{\Pgeq{k}}_\lambda$ for every $k$. To avoid confusion, we have defined $\bfN_\lambda^k$ only for $\lambda\in X^+$. On  the other hand, $\bfM^{\Pgeq{k}}_{\lambda}$ is defined for every $\lambda\in X$. This explains why in \cref{eq a=0} we cannot write, for example, $\bfN^3_{\lambda-\alpha_{23}}$ instead of $\bfM^{\Pgeq{3}}_{\lambda-\alpha_{23}}$. 
\end{remark}

\begin{lemma} \label{lemma c=0 new}
Let $\lambda = a\varpi_1 + b\varpi_2 + c\varpi_3 +d\varpi_4 \in X^+$ and suppose that $c=0$. Then
\begin{equation} \label{eq easy lemma new}
    \bfN_{\lambda}^2 = \bfN_{\lambda}^3 - q\bfMt_{\lambda - \alpha_{12}} -q^2\bfM_{\lambda-\alpha_{24} }^{\Pgeq{3}} +q^3\bfM_{\lambda -\alpha_{12} -\alpha_{24}}^{\Pgeq{3}}.
\end{equation}
\end{lemma}

\begin{lemma} \label{lemma zeros in the middle new}
Let $\lambda = a\varpi_1  +d\varpi_4 \in X^+$. Then:
\[    \bfN_{\lambda}^2 = \bfN_{\lambda}^3 - q^2\bfMt_{\lambda-\alpha_{14}} +q \bfM_{\lambda - \alpha_{23}}^{\Pgeq{3}} -q^3\bfM_{\lambda - \alpha_{14}-\alpha_{23}}^{\Pgeq{3}}.\]
\end{lemma}

\begin{lemma} \label{lemma zeros in the middle}
Let $\lambda = a\varpi_1  +d\varpi_4 \in X^+$.  Then, we have
\begin{equation*}
    \bfM_{\lambda - \alpha_{23}}^{\Pgeq{3} } = \left\{  \begin{array}{ll}
    0,& \mbox{if } \min(a,d)=0;\\
    -q^2\undH_{\lambda -\alpha_{14}} ,   & \mbox{if } \min(a,d)=1;  \\
    -q^2\undH_{\lambda -\alpha_{14}} +q^3\undH_{\lambda- 2\alpha_{14}}    ,& \mbox{otherwise. } 
    \end{array}  \right.
\end{equation*}
\end{lemma}

\begin{lemma} \label{lemma N3 zeros in the middle}
Let $\lambda = a\varpi_1  +d\varpi_4 \in X^+$. Then,
\begin{equation*}
    \bfN_{\lambda}^3 = \left\{  \begin{array}{rl}
       \undH_{\lambda} -q\undH_{\lambda - \alpha_{14}},   & \mbox{if } \min(a,d)>0;   \\
       \undH_{\lambda},  & \mbox{if } \min(a,d)=0. 
    \end{array}  \right.
\end{equation*}
\end{lemma}

\begin{prop}  \label{Propo N2 in N3 without hat}
   Let  $\lambda = a\varpi_1 + b\varpi_2 + c\varpi_3 +d\varpi_4 \in X^+$. In type $\tilde{A}_4$ the non-generic decomposition for $\bfN^2_{\lambda}$ in terms of $\{\bfN^3_{\mu} \mid  \mu \in X^+   \}$ is given in \Cref{tab:decom N2 in N3 for A4}.	  
\end{prop}
\begin{table}[h!]
    \centering
   
    \resizebox{\linewidth}{!}{
    $\begin{array}{||N||c|c|c|c||l||}\hline 
\multicolumn{1}{|c||}{\mbox{Row}} &    a     &    b     &    c    &    d      &  \qquad \qquad \qquad  \mbox{Decomposition}           \\ \hline  \hline
      \label{0110}     & 0        &  \geq 1  & \geq 1  &  0        &   \bfN_{\lambda}^3 -q \bfN_{\lambda - \alpha_{23}}^3         \\ \hline 
     \label{0200}     & 0        &  \geq 2  & 0       &  0        &  \bfN_{\lambda}^3 - q^2\bfN_{\lambda -\alpha_{12}-\alpha_{23} }^3           \\ \hline 
      \label{0121}     & 0        &  \geq 1  & \geq 2  &  \geq 1   &   \bfN_{\lambda}^3 - q\bfN_{\lambda -\alpha_{23}}^3 - q \bfN_{\lambda -\alpha_{34}}^3 +q^2\bfN_{\lambda -\alpha_{23}-\alpha_{34}}^3        \\ \hline  
      \label{0111}     & 0  	   &  \geq 1  & 1       &  \geq 1   &   \bfN_{\lambda}^3 -q\bfN_{\lambda - \alpha_{23}}^3-q\bfN_{\lambda - \alpha_{34}}^3 +q^3 \bfN_{\lambda -\alpha_{12}-\alpha_{23}-\alpha_{34}}^3             \\ \hline 
      \label{0201}     & 0   	   &  \geq 2  & 0       &  \geq 1   &  \bfN_{\lambda}^3 -q^2\bfN_{\lambda -\alpha_{24} }^3-q^2\bfN_{\lambda -\alpha_{12}-\alpha_{23} }^3  + q^3\bfN_{\lambda -\alpha_{12}-\alpha_{24} }^3          \\ \hline 
     \label{0102}     & 0        &   1      & 0       &  \geq 2   & \bfN_{\lambda}^3-q^2\bfN_{\lambda-\alpha_{24}}^3  +q^4\bfN_{\lambda -\alpha_{24} -\alpha_{14}}^3            \\ \hline 
     \label{0101}     & 0        &   1      & 0       &  1        &     \bfN_{\lambda}^3-q^2\bfN_{\lambda-\alpha_{24}}^3          \\ \hline 
      \label{0100}     & 0        &   1      & 0       &  0        &     \bfN_{\lambda}^3           \\ \hline 
     \label{0011}     & 0        &   0      & \geq 1  &  \geq 1   &  \bfN_{\lambda}^3  -q\bfN_{\lambda - \alpha_{34}}^3       \\ \hline  
     \label{0000}     & 0        &   0      & 0       &  \geq 0   &     \bfN_{\lambda}^3          \\ \hline 
      \label{1121}    & \geq 1   &   1      & \geq 2  &  \geq 1   &   \bfN_{\lambda}^3 -q\bfN_{\lambda - \alpha_{12}}^3 -q\bfN_{\lambda - \alpha_{23}}^3 -q\bfN_{\lambda -\alpha_{34}}^3 +q^2\bfN_{\lambda -\alpha_{12}-\alpha_{34}}^3 +q^2\bfN_{\lambda  -\alpha_{23} -\alpha_{34}}^3       \\ \hline  
      \label{1201}    & \geq 1   &  \geq 2  & 0       &  \geq 1   &  \bfN_{\lambda}^3 -q\bfN_{\lambda - \alpha_{12}}^3 -q^2\bfN_{\lambda- \alpha_{24} }^3 + q^3  \bfN_{\lambda- \alpha_{12}- \alpha_{24} }^3    \\ \hline 
      \label{1101}    & \geq 1   &   1      & 0       &  1        &  \bfN_{\lambda}^3 -q\bfN_{\lambda - \alpha_{12}}^3 -q^2 \bfN_{\lambda - \alpha_{24}}^3       \\ \hline 
      \label{1102}    & \geq 1   &   1      & 0       &  \geq 2   &  \bfN_{\lambda}^3 -q\bfN_{\lambda - \alpha_{12}}^3 -q^2  \bfN_{\lambda - \alpha_{24}}^3 +q^4   \bfN_{\lambda -\alpha_{14} - \alpha_{24}}^3        \\ \hline 
      \label{1001}    & \geq 1   &   0      & 0       &  1   &   \bfN_{\lambda}^3 -(q^2+q^3) \bfN_{\lambda -\alpha_{14}}^3          \\ \hline  
    \label{2002}    & \geq 2   &   0      & 0       &  \geq 2   &  \bfN_{\lambda}^3 -(q^2+q^3) \bfN_{\lambda -\alpha_{14}}^3  +q^5\bfN_{\lambda - 2\alpha_{14} }^3           \\ \hline
       \label{1111}    & \geq 1   &  1       & 1       &  \geq 1   &   \bfN_{\lambda}^3 -q\bfN_{\lambda - \alpha_{12}}^3 -q\bfN_{\lambda - \alpha_{23}}^3 -q\bfN_{\lambda -\alpha_{34}}^3 +q^2 \bfN_{\lambda -\alpha_{12} -\alpha_{34} }^3  +q^3  \bfN_{\lambda -\alpha_{12}-\alpha_{23}-\alpha_{34} }^3       \\ \hline 
  
      \label{0020}     & 0        &   0      & \geq 2  &  0        &  \bfN_{\lambda}^3  - q^2\bfN_{\lambda - \alpha_{23}-\alpha_{34}}^3        \\ \hline 
      \label{1210}    & \geq 1   &  \geq 2  & \geq 1  &  0        &   \bfN_{\lambda}^3 - q\bfN_{\lambda -\alpha_{23}}^3 - q \bfN_{\lambda -\alpha_{12}}^3 +q^2\bfN_{\lambda -\alpha_{12}-\alpha_{23}}^3         \\ \hline 
      \label{1110}    & \geq 1   &   1      & \geq 1  &  0        &  \bfN_{\lambda}^3 -q\bfN_{\lambda - \alpha_{23}}^3-q\bfN_{\lambda - \alpha_{12}}^3 +q^3 \bfN_{\lambda -\alpha_{12}-\alpha_{23}-\alpha_{34}}^3           \\ \hline
        \label{1020}    & \geq 1   &   0      & \geq 2  &  0        &  \bfN_{\lambda}^3 -q^2\bfN_{\lambda -\alpha_{13} }^3-q^2\bfN_{\lambda -\alpha_{23}-\alpha_{34} }^3  + q^3\bfN_{\lambda -\alpha_{34}-\alpha_{13} }^3         \\ \hline 
            \label{2010}    & \geq 2   &   0      & 1       &  0        &   \bfN_{\lambda}^3 -q^2\bfN_{\lambda - \alpha_{13}}^3 +q^4 \bfN_{\lambda - \alpha_{13} -\alpha_{14}}^3         \\ \hline 
        \label{1010}    &  1       &   0      & 1       &  0        &  \bfN_{\lambda}^3 -q^2\bfN_{\lambda - \alpha_{13}}^3                 \\ \hline 
        \label{0010}     & 0        &   0      & 1       &  0        &     \bfN_{\lambda}^3               \\ \hline 

         \label{1100}    & \geq 1   &  \geq 1  & 0       &  0        &  \bfN_{\lambda}^3 -q\bfN_{\lambda - \alpha_{12}}^3        \\ \hline 
      \label{1000}    & \geq 1   &   0      & 0       &  0   &   \bfN_{\lambda}^3         \\ \hline      
    \label{1211}    & \geq 1   &  \geq 2  & 1       &  \geq 1   &  \bfN_{\lambda}^3 -q\bfN_{\lambda - \alpha_{12}}^3 -q\bfN_{\lambda - \alpha_{23}}^3 -q\bfN_{\lambda -\alpha_{34}}^3 +q^2 \bfN_{\lambda -\alpha_{12} -\alpha_{34} }^3  +q^2  \bfN_{\lambda -\alpha_{12}-\alpha_{23} }^3           \\ \hline 
    \label{1021}    & \geq 1   &   0      & \geq 2  &  \geq 1   &  \bfN_{\lambda}^3 -q\bfN_{\lambda - \alpha_{34}}^3 -q^2  \bfN_{\lambda - \alpha_{13}}^3  +q^3 \bfN_{\lambda - \alpha_{13}-\alpha_{34} }^3         \\ \hline    
    \label{1011}    &   1      &   0      &  1      &  \geq 1   & \bfN_{\lambda}^3 -q\bfN_{\lambda - \alpha_{34}}^3 -q^2  \bfN_{\lambda - \alpha_{13}}^3      \\ \hline
     \label{2011}    & \geq 2   &   0      &  1      &  \geq 1   & \bfN_{\lambda}^3 -q\bfN_{\lambda - \alpha_{34}}^3 -q^2  \bfN_{\lambda - \alpha_{13}}^3  +q^4 \bfN_{\lambda - \alpha_{13}-\alpha_{14} }^3       \\ \hline  
    \label{1001bis}    &  1       &   0      & 0       &  \geq 1   &  \bfN_{\lambda}^3 -(q^2+q^3) \bfN_{\lambda -\alpha_{14}}^3            \\ \hline

       \end{array}$
      }
  
    \caption{Decomposition of $\bfN^2_{\lambda} $ in terms of $\{\bfN_{\mu}^3 \mid \mu \in X^+\} $ for $\tilde{A}_4$. }
    \label{tab:decom N2 in N3 for A4}
\end{table}

\begin{proof}
The proof follows by a  case-by-case analysis. Notice that Rows 18-31 are symmetric to Rows 2-15. So we need to consider only the first 17 rows.
We will make use of the following sets: \[B:=\Pgeq{2}\setminus \{ \alpha_{12},\alpha_{34}\},\; C:=\Pgeq{3}\setminus \{\alpha_{13}\}\ \mathrm{and }\;D:=\Pgeq{3} \setminus \{\alpha_{24}\}.\]
Notice that $B=s_1(B)=s_4(B)$, $C=s_1(C)=s_3(C)$ and $D=s_2(D)=s_4(D)$.
\begin{enumerate}[label=\textbf{R\arabic{enumi}.},ref=1.\arabic{enumi}, wide=0pt]

\item[\textbf{R\ref{0110}.}]   We have $\bfN_{\lambda }^2 = \bfM_{\lambda}^B -q \bfM_{\lambda - \alpha_{12}}^B -q \bfM_{\lambda - \alpha_{34}}^B +q^2\bfM_{\lambda - \alpha_{12}-\alpha_{34}}^B $.
Since $s_1(B)=s_4(B)=B$ and  $a=d=0$ we can use \Cref{lemma X equal to zero} to conclude that	$\bfM_{\lambda - \alpha_{12}}^B = \bfM_{\lambda - \alpha_{34}}^B =\bfM_{\lambda - \alpha_{12}-\alpha_{34}}^B =0$. Therefore, 
$\bfN_{\lambda }^2 = \bfM_{\lambda}^B  = \bfN_{\lambda}^3 - q\bfN_{\lambda - \alpha_{23}}^3$.
\item[\textbf{R\ref{0200}.}]  Arguing as in the proof of Row \ref{0110} we arrive to 
$\bfN_{\lambda }^2 = \bfM_{\lambda}^B  = \bfN_{\lambda}^3 - q\bfM_{\lambda - \alpha_{23}}^{\Pgeq{3}}$. Using set $C$ to decompose the elements $\bfM_{\lambda - \alpha_{23}}^{\Pgeq{3}}$ and $ \bfN_{\lambda - \alpha_{12}-\alpha_{23}}^3$ and applying \Cref{lemma X equal to zero} we obtain $\bfM_{\lambda - \alpha_{23}}^{\Pgeq{3}} = q\bfN_{\lambda - \alpha_{12}-\alpha_{23}}^3$. Therefore, $\bfN_{\lambda }^2   = \bfN_{\lambda}^3 - q^2\bfN_{\lambda - \alpha_{12}-\alpha_{23}}^3$.

\item[\textbf{R\ref{0121}.}]
This case is a direct consequence of \Cref{lemma a=0}. 

\item[\textbf{R\ref{0111}.}] By \Cref{lemma a=0} we have
\begin{equation} \label{eq decom A4 case three}
	\bfN_{\lambda}^2 =\bfN_{\lambda}^3 - q\bfN_{\lambda -\alpha_{23}}^3 - q \bfN_{\lambda -\alpha_{34}}^3 +q^2\bfM_{\lambda -\alpha_{23}-\alpha_{34}}^{\Pgeq{3}}. 
\end{equation}
We have $s_3(C)=C$, thus \Cref{lemma X equal to zero} implies  that $\bfM_{\lambda -\alpha_{23}-\alpha_{34}}^{\Pgeq{3}}  = q\bfN_{\lambda -\alpha_{12}- \alpha_{23}-\alpha_{34}}^3$. By plugging this in \cref{eq decom A4 case three} we obtain the desired decomposition. 
\item[\textbf{R\ref{0201}.}] \Cref{lemma a=0} implies
\begin{equation} \label{eq decom A4 case four A}
	\bfN_{\lambda}^2 =\bfN_{\lambda}^3 - q\bfM_{\lambda -\alpha_{23}}^{\Pgeq{3}} - q \bfM_{\lambda -\alpha_{34}}^{\Pgeq{3}} +q^2\bfM_{\lambda -\alpha_{23}-\alpha_{34}}^{\Pgeq{3}}. 
\end{equation}
Arguing as in the previous case, using the fact that $s_1(C)=s_3(C)=C$, we obtain
\begin{equation} \label{eq decom A4 case four B}
\bfM_{\lambda -\alpha_{23}}^{\Pgeq{3}} =q\bfN_{\lambda -\alpha_{12}-\alpha_{23}}^3, \quad  \bfM_{\lambda -\alpha_{34}}^{\Pgeq{3}}=0 ,\quad 	\bfM_{\lambda -\alpha_{23}-\alpha_{34}}^{\Pgeq{3}} = - \bfN_{\lambda - \alpha_{24}}^3 +q \bfN_{\lambda -\alpha_{12}-\alpha_{24}}^3. 
\end{equation}
Then plugging \cref{eq decom A4 case four B} in \cref{eq decom A4 case four A} we obtain the desired decomposition. 

\item[\textbf{R\ref{0102}.}]  This case is similar to Row \ref{0201}. 
\item[\textbf{R\ref{0101}, R\ref{0100}.}] These rows provide a specific value for $\lambda$. Therefore, they follow by a direct computation. 

\item[\textbf{R\ref{0011}.}] By \Cref{lemma a=0} we have
\begin{equation} \label{eq row 13 and 15}
	\bfN_{\lambda}^2 = \bfN_{\lambda }^3 - q\bfN_{\lambda - \alpha_{34}}^3 -q\bfM_{\lambda -\alpha_{23} }^{\Pgeq{3}} +q^2\bfM_{\lambda -\alpha_{23}-\alpha_{34}}^{\Pgeq{3}}.  
\end{equation}  
We use the set $D$ to expand the $\bfM$-elements above. Since $s_2(D)=D$, \Cref{lemma X equal to zero} shows that \cref{eq row 13 and 15} reduces to the decomposition predicted by the table. 
\item[\textbf{R\ref{0000}.}]  Arguing as in the proof of Row \ref{0011} we obtain 
$	\bfN_{\lambda}^2 = \bfN_{\lambda }^3 - q\bfM_{\lambda - \alpha_{34}}^{\Pgeq{3}}$ .  
Using \Cref{lemma X equal to zero} and the fact that $s_1(C)=s_3(C)=C$ we obtain
$
\bfM_{\lambda - \alpha_{34}}^{\Pgeq{3}} = 	\bfM_{\lambda - \alpha_{34}}^C -q \bfM_{\lambda - \alpha_{34}-\alpha_{13}}^C =0
$.
\item[\textbf{R\ref{1121}.}] By definition of $\bfN_{\lambda}^2$ we have
\begin{equation*}
\bfN_{\lambda}^2  = \bfN_{\lambda}^3 -q\bfN_{\lambda - \alpha_{12}}^3 -q\bfN_{\lambda - \alpha_{23}}^3 -q\bfN_{\lambda -\alpha_{34}}^3 +q^2\bfN_{\lambda -\alpha_{12}-\alpha_{34}}^3 +q^2\bfN_{\lambda  -\alpha_{23} -\alpha_{34}}^3 + q^2Z,
\end{equation*}
where $Z= \bfM_{\lambda -\alpha_{12}-\alpha_{23} }^{\Pgeq{3}} -  q\bfN_{\lambda - \alpha_{12}-\alpha_{23}-\alpha_{34}}^3 $. Using the set $D$ to further decompose $Z$,  since $s_2(D)=D$ we obtain $Z=0$. 

\item[\textbf{R\ref{1201}.}]  This is a direct consequence of \Cref{lemma c=0 new}.
\item[\textbf{R\ref{1101}.}] Using \Cref{lemma c=0 new}, it remains to show that $\bfM_{\lambda - \alpha_{12}-\alpha_{24}}^{\Pgeq{3}}=0$. This follows after decomposing it using $D$, since $s_2(D)=s_4(D)=D$.
\item[\textbf{R\ref{1102}.}] Using \Cref{lemma c=0 new}, it remains to show that $\bfM_{\lambda - \alpha_{12}-\alpha_{24}}^{\Pgeq{3}}=q\bfN^3_{\lambda-\alpha_{14}-\alpha_{24}}$. This follows after decomposing both sides using $D$, since $s_2(D)=D$.
 \item[\textbf{R\ref{1001}.}]  Using \Cref{lemma zeros in the middle new} and \Cref{lemma zeros in the middle} we obtain
$ \bfN_{\lambda}^2 =\bfN_{\lambda}^3 - q^2\bfN_{\lambda-\alpha_{14}}^3 -q^3 \undH_{\lambda - \alpha_{14}}$.
We conclude by noticing that \Cref{lemma N3 zeros in the middle} implies that $\bfN_{\lambda -\alpha_{14}}^3 = \undH_{\lambda - \alpha_{14}}$.


\item[\textbf{R\ref{2002}.}]  Let us first assume that $\min (a,d)>2$. Arguing as in Row \ref{1001}, by \Cref{lemma zeros in the middle new,lemma zeros in the middle} we obtain
\begin{equation} \label{eq row 31}
    \bfN_{\lambda}^2 = \bfN_{\lambda}^3 - q^2\bfN_{\lambda-\alpha_{14}}^3 -q^3\undH_{\lambda - \alpha_{14}} +q^4\undH_{\lambda - 2\alpha_{14}} +q^5 \undH_{\lambda - 2\alpha_{14}} -q^6\undH_{\lambda - 3\alpha_{14}}.
\end{equation}
Then, \Cref{lemma N3 zeros in the middle} implies that $\bfN_{\lambda}^2 = \bfN_{\lambda}^3 -( q^2+q^3)\bfN_{\lambda-\alpha_{14}}^3 +q^5\bfN_{\lambda - 2\alpha_{14}}^3$, which is the desired decomposition. Assume now  $\min (a,d)=2$. In this case, the term $q^6\undH_{\lambda - 3\alpha_{14}}$ does not appear in  \cref{eq row 31}. However, Lemma  \ref{lemma N3 zeros in the middle} also implies the desired decomposition in this case. 
\item[\textbf{R\ref{1111}.}] By definition of $\bfN_{\lambda}^2$ we have
\begin{equation} \label{eq row 16}
    \bfN_{\lambda}^2 = \bfN_{\lambda}^3 -q\bfN_{\lambda -\alpha_{12}}^3 -q\bfN_{\lambda -\alpha_{23}}^3- q\bfN_{\lambda- \alpha_{34}}^3 + q^2\bfN_{\lambda -\alpha_{12}-\alpha_{34} }^3-q^3\bfN_{\lambda-\alpha_{12}-\alpha_{23}-\alpha_{34}}^3 + q^2Y, 
\end{equation}
where $Y = \bfM^{\Pgeq{3}}_{\lambda-\alpha_{12}-\alpha_{23}} + \bfM^{\Pgeq{3}}_{\lambda-\alpha_{23}-\alpha_{34}} $. Using the set $D$, we obtain
$\bfM^{\Pgeq{3}}_{\lambda-\alpha_{12}-\alpha_{23}}=q\bfN^{3}_{\lambda-\alpha_{12}-\alpha_{23}-\alpha_{34}}$.  By symmetry, we also have $\bfM^{\Pgeq{3}}_{\lambda-\alpha_{23}-\alpha_{34}}=q\bfN^3_{\lambda-\alpha_{12}-\alpha_{23}-\alpha_{34}}$. 
Therefore, $Y=2q\bfN_{\lambda-\alpha_{12}-\alpha_{23}-\alpha_{34}}^3 $ and \cref{eq row 16} reduces to the desired decomposition. 

\end{enumerate}
Having checked all 17 cases, the proof is complete.
\end{proof}

We now move to proving ``hat'' version of the proposition above. As before, we start by recording in the following Lemmas some useful computations. Both lemmas  follow using the definition of admissibility.  We leave the proofs to the reader.

\begin{lemma} \label{lemma N3 en N2 two zeros in the middle}
Let $\lambda = a\varpi_1  +d\varpi_4 \in X^+$. Then, 
\begin{equation*}
    \hat{\bfN}_\lambda^3 = \displaystyle \sum_{j=0}^{\min(a,d)} \left( q^{2j}\sum_{k=0}^jq^k \right) \bfN_{\lambda -j\alpha_{14}}^2.
\end{equation*}
\end{lemma}


\begin{lemma} \label{lemma c=0 minus alpha12}
Let $\lambda = a\varpi_1  +b\varpi_2+  d\varpi_4 \in X^+$. 
\begin{enumerate}
    \item \label{516a}
If  $a,b\geq 1$ then
\begin{equation}\label{eq lemma c=0 minus a12}
    \hat{\bfN}^3_{\lambda} -q\hat{\bfN}^3_{\lambda -\alpha_{12}} = \left\{ 
    \begin{array}{ll}
     \displaystyle \sum_{i=0}^d q^{2i} \bfN_{\lambda - i\alpha_{24}}^2         &\mbox{if } b\geq d; \\
       &  \\  
\displaystyle \sum_{i=0}^b q^{2i} \bfN_{\lambda - i\alpha_{24}}^2 + \sum_{j=1}^{\min(a+b,d-b)} q^{2b+3j}\bfN_{\lambda-b\alpha_{24}-j\alpha_{14}}^2 ,         & \mbox{if } b< d.
    \end{array}  
        \right. 
\end{equation}
\item \label{516b} If  $a=0$, $b\geq 2$ and $d\geq 1$ then $\hat{\bfN}^3_{\lambda} -q^2\hat{\bfN}^3_{\lambda -\alpha_{12}-\alpha_{23} } $  is equal to the right-hand side of \cref{eq lemma c=0 minus a12}.
\end{enumerate}
\end{lemma}

\begin{prop} \label{propo N3 in terms of N2 with hat}
  Let  $\lambda = a\varpi_1 + b\varpi_2 + c\varpi_3 +d\varpi_4 \in X^+$. In type $\tilde{A}_4$ the non-generic decomposition for $\bfN^2_{\lambda}$ in terms of $\{\hat{\bfN}^3_{\mu} \mid  \mu \in X^+   \}$ is given in \Cref{tab:decom N2 in N3 for A4}  (replacing $\bfN^3_\mu$ with $\hat{\bfN}^3_\mu$).	
\end{prop}

\begin{proof}
As in the proof of \Cref{Propo N2 in N3 without hat} we only need to consider the first 17 rows of \Cref{tab:decom N2 in N3 for A4}. We first notice that the argument given in the proof of \Cref{Propo generic case A4} carries over for the decomposition in Rows \ref{0110}, \ref{0200}, \ref{0121}, \ref{0111}, \ref{0011}, \ref{1121} and \ref{1111}. 
On the other hand, Rows \ref{0101} and \ref{0100}  provide a specific value for $\lambda$ and therefore these cases follow by a direct computation. Furthermore, \Cref{lemma N3 en N2 two zeros in the middle} gives us the result for Rows \ref{0000}, \ref{1001} and \ref{2002}. This leaves us with five cases to be checked. 

\begin{enumerate}[label=\textbf{R\arabic{enumi}.},ref=1.\arabic{enumi}, wide=0pt]
    \item[\textbf{R\ref{0201}.}] Using  \Cref{lemma c=0 minus alpha12}\eqref{516b} for $\lambda$ and \Cref{lemma c=0 minus alpha12}\eqref{516a} for $\lambda -\alpha_{24}$ we obtain, as required: 
    \begin{equation*}
       \bfN_{\lambda}^2 = \left(  \hat{\bfN}_{\lambda}^3 -q\hat{\bfN}_{\lambda - \alpha_{12}-\alpha_{23} }^3  \right) -q^2\left( \hat{\bfN}_{\lambda -\alpha_{24}}^3 -q\hat{\bfN}_{\lambda -\alpha_{24} - \alpha_{12}}^3   \right).
    \end{equation*}
    \item[\textbf{R\ref{0102}}.]  Using the definition of $\hat{\bfN}_{\lambda}^3$ and noticing which $\mu$ are such that $\mathcal{L}^a_{\lambda}(\mu)\neq \emptyset,$  one obtains
    \begin{equation} \label{eq row 9 hat A}
        \hat{\bfN}_{\lambda}^3 = \bfN_{\lambda}^2 +q^2\bfN_{\lambda - \alpha_{24}}^2 +q^5\bfN_{\lambda-\alpha_{24}-\alpha_{14}}^2.
    \end{equation}
    On the other hand, using \Cref{lemma N3 en N2 two zeros in the middle} twice, we obtain
    \begin{equation}\label{eq row9 hat B}
      \hat{\bfN}_{\lambda - \alpha_{24}}^3 -q^2\hat{\bfN}_{\lambda - \alpha_{24}-\alpha_{14}}^3 = \bfN_{\lambda -\alpha_{24}}^2 +q^3\bfN_{{\lambda -\alpha_{24}-\alpha_{14}} }^2.   
    \end{equation}
    Combining \cref{eq row 9 hat A} with \cref{eq row9 hat B} we get the desired decomposition.
    \item[\textbf{R\ref{1201}.}]  Using \Cref{lemma c=0 minus alpha12}\eqref{516a} for $\lambda$ and $\lambda -\alpha_{24}$ we obtain the required decomposition
    \begin{equation*}
     \bfN_{\lambda}^2  = \left(  \hat{\bfN}_{\lambda}^3 -q\hat{\bfN}_{\lambda - \alpha_{12}}^3  \right) -q^2\left( \hat{\bfN}_{\lambda -\alpha_{24}}^3 -q\hat{\bfN}_{\lambda -\alpha_{24} - \alpha_{12}}^3   \right).
    \end{equation*}
    \item[\textbf{R\ref{1101}.}] By \Cref{lemma c=0 minus alpha12}\eqref{516a} we have $
        \hat{\bfN}_{\lambda}^3 -q\hat{\bfN}_{\lambda -\alpha_{12}}^3 = \bfN_{\lambda}^2 +q^2\bfN_{\lambda - \alpha_{24}}^2$. 
    Since $\lambda -\alpha_{24} = (a+1)\varpi_1 $, \Cref{lemma N3 en N2 two zeros in the middle} implies $\hat{\bfN}_{\lambda - \alpha_{24}}^3 = \bfN_{\lambda - \alpha_{24}}^2 $. We conclude that
$ \bfN_{\lambda}^2 = \hat{\bfN}_{\lambda}^3 -q\hat{\bfN}_{\lambda -\alpha_{12}}^3 -  q^2\hat{\bfN}_{\lambda - \alpha_{24}}^3 $, 
as we wanted to show.  
    \item[\textbf{R\ref{1102}.}]  Using \Cref{lemma c=0 minus alpha12}\eqref{516a} we obtain
    \begin{equation} \label{eq row 23 hat A}
        \hat{\bfN}_{\lambda}^3 -q\hat{\bfN}_{\lambda -\alpha_{12}}^3  = \bfN_{\lambda}^2+q^2\displaystyle \sum_{j=0}^{\min(a+1,d-1)} q^{3j} \bfN_{\lambda -\alpha_{24}- j\alpha_{14}}^2 .
    \end{equation}
    On the other hand, \Cref{lemma N3 en N2 two zeros in the middle} implies
    \begin{equation}\label{eq row 23 hat B}
        \hat{\bfN}_{\lambda -\alpha_{24}}^3 -q^2\hat{\bfN}_{\lambda - \alpha_{24}-\alpha_{14}}^3 =   \displaystyle \sum_{j=0}^{\min(a+1,d-1)} q^{3j} \bfN_{\lambda -\alpha_{24}- j\alpha_{14}}^2 .
    \end{equation}
   By combining \cref{eq row 23 hat A} with \cref{eq row 23 hat B} we obtain the desired decomposition. 
\end{enumerate}
This finishes the proof of the Proposition.
\end{proof}

\begin{proof}[Proof of \Cref{teo decomposition in A4}]
The result follows by combining  \Cref{Propo generic case A4}, \Cref{Propo N2 in N3 without hat} and \Cref{propo N3 in terms of N2 with hat}.
\end{proof}

{\footnotesize
\bibliography{mybiblio} 
\bibliographystyle{alpha}
}

\newpage

\end{document}